\documentclass[10pt]{article}

\usepackage{amsmath,amsfonts,amssymb} 

\usepackage{amsthm} 

\usepackage{graphicx,caption,subcaption} 
\usepackage{fullpage}

\allowdisplaybreaks 

\begingroup
    \makeatletter
    \@for\theoremstyle:=definition,remark,plain\do{%
        \expandafter\g@addto@macro\csname th@\theoremstyle\endcsname{%
            \addtolength\thm@preskip\parskip
            }%
        }
\endgroup

\newcommand{\wt}{\widetilde}

\newtheorem{thm}{Theorem}
\newtheorem{lemma}{Lemma}[section]
\newtheorem{cor}[lemma]{Corollary}

\newtheorem{prop}[thm]{Proposition}
\newtheorem{coro}[thm]{Corollary}
\newtheorem{construction}{Construction}[section]

\newtheorem{claim}[lemma]{Claim}

\newtheorem{definition}[lemma]{Definition}

\title{}
\author{}

\begin{document}

\title{\vspace{-0.5in} A hypergraph analog \\
of Dirac's Theorem for long cycles in 2-connected graphs}

\author{
{{Alexandr Kostochka}}\thanks{
\footnotesize {University of Illinois at Urbana--Champaign, Urbana, IL 61801. E-mail: \texttt {kostochk@illinois.edu}.
 Research 
is supported in part by  NSF  Grant DMS-2153507.
}}
\and
{{Ruth Luo}}\thanks{
\footnotesize {University of South Carolina, Columbia, SC 29208, USA. E-mail: \texttt {ruthluo@sc.edu}.
}}
\and{{Grace McCourt}}\thanks{University of Illinois at Urbana--Champaign, Urbana, IL 61801, USA. E-mail: {\tt mccourt4@illinois.edu}. Research 
is supported in part by NSF RTG grant DMS-1937241.}}

\date{ \today}
\maketitle

\vspace{-0.3in}

\begin{abstract}
Dirac proved that each $n$-vertex $2$-connected graph with minimum degree at least $k$ contains a cycle  of length at least $\min\{2k, n\}$. We consider a hypergraph version of this result. A {\em Berge cycle} in a hypergraph is an alternating sequence of distinct vertices and edges $v_1,e_2,v_2, \ldots, e_c, v_1$ such that $\{v_i,v_{i+1}\} \subseteq e_i$ for all $i$ (with indices taken modulo $c$). We prove that for $n \geq k \geq r+2 \geq 5$, every $2$-connected $r$-uniform $n$-vertex hypergraph with minimum degree at least ${k-1 \choose r-1} + 1$ has a Berge cycle of length at least $\min\{2k, n\}$. The bound is exact for all
$k\geq r+2\geq 5$.

\end{abstract}

\section{Introduction and Results}
\subsection{Terminology and known results for graphs}
 A hypergraph $H$ is a family of subsets of a ground set. We refer to these subsets as the {\em edges} of $H$ and the elements of the ground set as the {\em vertices} of $H$. We use $E(H)$ and $V(H)$ to denote the set of edges and the set of vertices of $H$ respectively. We say $H$ is {\em $r$-uniform}  ($r$-graph, for short) if every edge of $H$ contains exactly $r$ vertices. A {\em graph} is a 2-graph. 
 For a hypergraph $H$ and $A\subseteq V(H)$, by $H[A]$ we denote the subhypergraph of $H$ induced by $A$.

The {\em degree} $d_H(v)$ of a vertex $v$ in a hypergraph $H$ is the number of edges containing $v$. When there is no ambiguity, we may drop the subscript $H$ and simply use $d(v)$. The {\em minimum degree}, $\delta(H)$, is the minimum over degrees of all vertices of $H$. 

A {\em hamiltonian cycle} in a graph is a cycle which visits every vertex. 
 Sufficient conditions for existence of hamiltonian cycles in graphs have been well-studied. In particular, a famous result of  Dirac from 1952 is: 

\begin{thm}[Dirac~\cite{D}]\label{dirac}
Let $n \geq 3$. If $G$ is an $n$-vertex graph with minimum degree $\delta(G) \geq n/2$, then $G$ has a hamiltonian cycle.
\end{thm}

Dirac also proved that every graph $G$ with minimum degree $k \geq 2$ contains a cycle of length at least $k+1$, and that this bound can be significantly strengthened when $G$ is 2-connected.
\begin{thm}[Dirac~\cite{D}]\label{dirac2}
Let $n \geq k \geq 2$. If $G$ is an $n$-vertex, 2-connected graph with minimum degree $\delta(G) \geq k$, then $G$ has a cycle of length at least $\min\{2k, n\}$.
\end{thm}

This theorem is sharp by the following examples. First, for $k\geq 3$, let $V(G_1) = X_1\cup X_2\cup \ldots\cup X_t$ where $|X_i| = k$ and
 $G_1[X_i]=K_k$ 
 for all $1\leq i\leq t$, and there are vertices $u,v$ such that $X_i \cap X_j = \{u_, v\}$ for all $i\neq j$. 
 Since $k\geq 3$, $\delta(G_1)=k-1$, and each cycle in $G_1$ intersects at most $2$ sets $X_i\setminus \{u,v\}$, thus  having length at most $k + k - 2 =2k-2$. Another example is the graph $G_2$ obtained by joining every vertex of the clique $K_{k-1}$ to every vertex of an independent set with $n-(k-1)$ vertices.  Again, $\delta(G_2)=k-1$, 
 and each cycle in $G_2$ has length at most $2(k-1)  = 2k-2$. 
Moreover,  $G_2$ is $(k-1)$-connected. So for $k$ large, one cannot improve the bound in Theorem~\ref{dirac2} 
 by requiring higher connectivity. 

A refinement of Theorem~\ref{dirac2} for bipartite graphs  was obtained by Voss and Zuluaga~\cite{VZ}, which was further refined by Jackson~\cite{jackson2} as follows.

\begin{thm}[Jackson~\cite{jackson2}]\label{jack}
Let  $G$ be a 2-connected bipartite graph with bipartition $(A,B)$, where $|A|\geq |B|$. If each vertex of $A$ has degree at least $a$ and each vertex of $B$ has degree at least $b$,
 then $G$ has a cycle of length at least $2\min\{|B|, a+b-1, 2a-2\}$. Moreover, if $a=b$ and $|A|=|B|$, then $G$ has a cycle of length at least $2\min\{|B|,  2a-1\}$.
\end{thm}

A sharpness example for Theorem~\ref{jack} is a graph $G_3=G_3(a,b,a',b')$ for $a'\geq b'\geq a+b-1$ obtained from
disjoint complete bipartite graphs $K_{a'-b,a}$ and $K_{b,b'-a}$ by joining each vertex in the $a$ part of $K_{a'-b,a}$ to
each vertex in the $b$ part of $K_{b,b'-a}$.


\subsection{Terminology and known results for uniform hypergraphs}

We consider the notion of {\em Berge cycles}.

\begin{definition}
A {\bf Berge cycle of} length $c$ in a hypergraph is an alternating list of $c$ distinct vertices and $c$ distinct edges $C=v_1, e_1, v_2, \ldots,e_{c-1}, v_c, e_c, v_1$ such that $\{v_i, v_{i+1}\} \subseteq e_i$ for all $1\leq i \leq c$ (we always take indices of cycles of length $c$ modulo $c$).
We call vertices $v_1, \ldots, v_c$ {\bf the defining vertices} of $C$ and write
$V(C)=\{v_1, \ldots, v_c\}$, $E(C) = \{e_1, \ldots, e_c\}$. 
\end{definition}

Notation for  Berge paths is similar. In addition, a {\bf partial Berge path} is an alternating sequence of distinct  edges and vertices beginning with an edge and ending with a vertex $e_0, v_1, e_1, v_2, \ldots, e_k, v_{k+1}$ such that $v_1 \in e_0$ and for all $1\leq i \leq k$, $\{v_i, v_{i+1}\} \subseteq e_i$.

A series of approximations and analogs of Theorem~\ref{dirac} for Berge cycles in a number of classes of $r$-uniform hypergraphs ({\em $r$-graphs}, for short) were obtained by
Bermond,  Germa,  Heydemann and Sotteau~\cite{BGHS}, Clemens, Ehrenm\"uller and Person~\cite{CEP},
Coulson and Perarnau~\cite{CP} and Ma, Hou, and Gao~\cite{MHG}.

Exact bounds for all values of $3\leq r<n$ were obtained in~\cite{KLM}.

\begin{thm}[Theorem 1.7 in~\cite{KLM}]\label{mainold2}
Let $t=t(n)=\lfloor \frac{n-1}{2} \rfloor$, and suppose $3 \leq  r <n$. Let $H$ be an $r$-graph. If
\hspace{1mm}
(a) $r\leq t$ and $\delta(H) \geq  {t \choose r-1} +1$ or 
\hspace{0.5mm}
(b) $r\geq n/2$ and $\delta(H) \geq r$, 
 \hspace{0.5mm}
 then $H$ contains a hamiltonian Berge cycle. 
\end{thm}

Salia~\cite{SN} proved an exact result of P\' osa type extending  Theorem~\ref{mainold2} for $n>2r$ to  hypergraphs with ``few" vertices of small degree.
In~\cite{KLM}, some bounds on the circumference of $r$-graphs with given minimum degree were obtained:

\begin{thm}[\cite{KLM}]\label{oldmain} Let $n, k,$ and $r$ be positive integers such that $n \geq k$ and $\lfloor(n-1)/2 \rfloor  \geq r \geq 3$. Let $H$ be an $n$-vertex, $r$-uniform hypergraph. If 

\vspace{-3mm}

\begin{enumerate}
\item[(a)] $k \leq r+1$ and $\delta(H) \geq k-1$, or

\vspace{-2mm}

\item[(b)] $r+2\leq k<\lfloor(n-1)/2 \rfloor +2$ and $\delta(H) \geq {k-2 \choose r-1}+1$, or

\vspace{-2mm}

\item[(c)] $k\geq \lfloor(n-1)/2 \rfloor +2$ and $\delta(H) \geq {\lfloor(n-1)/2 \rfloor  \choose r-1}  + 1$,
\end{enumerate}
\vspace{-3mm}
then $H$ contains a Berge cycle of length $k$ or longer.
\end{thm}



For an analog of Theorem~\ref{dirac2}, we define connectivity of a hypergraph with the help of its {\em incidence bipartite graph}:

\begin{definition}Let $H$ be a hypergraph. The {\bf incidence  graph $I_H$ of $H$} is the bipartite graph with $V(I_H) = X \cup Y$ such that $X = V(H), Y = E(H)$ and for $x \in X, y\in Y$, $xy \in E(I_H)$ if and only if the vertex $x$ belongs to the edge $y$ in $H$.
\end{definition}

It is easy to see that if $H$ is an $r$-graph with minimum degree $\delta(H)$, then each $x \in X$ and each $y \in Y$ satisfy $d_{I_H}(x) \geq \delta(H), d_{I_H}(y) = r$. Moreover, there is a bijection between the set of Berge cycles of length $c$ in $H$ and the set of cycles of length $2c$ in $I_H$:   a Berge cycle $v_1, e_1, \ldots, v_c, e_c, v_1$ can also be viewed as a cycle in $I_H$ with the same sequence of vertices.

Using the notion of the incidence  graph, we also define connectivity in hypergraphs.

\begin{definition}
A hypergraph $H$ is {\bf $k$-connected} if its incidence  graph $I_H$ is a $k$-connected graph. 
\end{definition}

Theorem~\ref{jack} of Jackson applied to $I_H$ of a $2$-connected $r$-graph $H$ yields the following approximation of an analog of Theorem~\ref{dirac2}  for 
 $k\leq r-1$:

\begin{coro}\label{jackcor} Let $n, k, r$ be positive integers with $2 \leq k \leq r-1$. If $H$ is an  $n$-vertex $2$-connected $r$-graph $H$ with $\delta(H)\geq k+1$,
then $H$ contains a Berge cycle of length at least $\min\{2k,n,|E(H)|\}$. 
\end{coro}

On the other hand, for all $3\leq k\leq r$, there are  $2$-connected $r$-graphs $H_k$ with $\delta(H_k)\geq k-2$
that do not  have Berge cycle of length at least $\min\{2k,|V(H_k)|,|E(H_k)|\}$. A series of such examples is as follows. For $m\geq 2$, let 
$V(H_k)=A_1\cup \ldots \cup A_m\cup \{x,y\}$ where $A_i=\{a_{i,1},\ldots,a_{i,r-1}\}$ for $1\leq i\leq m$, and let
$E(H_k)=E_1\cup \ldots \cup E_m$ where for each $1\leq i\leq m$ and $1\leq j\leq k-1$, $E_i=\{e_{i,1},\ldots,e_{i,k-1}\}$ and
$e_{i,j}=(A_i-a_{i,j})\cup \{x,y\} $. Each Berge cycle in $H_k$ can contain edges from at most two $E_i$s, and $|E_i|=k-1$ for all $1\leq i\leq m$.

\subsection{Our results and structure of the paper}

Our main result is the following.
\begin{thm}\label{mainthm} Let $n, k, r$ be positive integers with $3\leq r \leq k-2\leq n-2 $. If $H$ is an  $n$-vertex $2$-connected $r$-graph with
\begin{equation}\label{main}
\delta(H) \geq {k-1 \choose r-1} + 1,
\end{equation}
then $H$ contains a Berge cycle of length at least $\min\{2k,n\}$. 
\end{thm}

We point out that for $2$-connected hypergraphs, the minimum degree required to guarantee a Berge cycle of length at least $2k$ is roughly of the order $2^{r-1}$ times smaller than the sharp bound guaranteed in Theorem~\ref{oldmain}(b). Furthermore, the bound $\delta(H) = {k-1 \choose r-1} + 1$  is best possible as demonstrated by the following constructions.

\begin{construction}Let $q\geq 2$ be an integer and $4 \leq r+1 \leq k \leq n/2$. For $n = q(k-2)+ 2$, let $H_1 = H_1(k)$ be the $r$-graph with $V(H_1) = \{x,y\} \cup V_1 \cup V_2 \cup \ldots \cup V_q$ where  for all $1\leq i\leq q$,  $|V_i|=k-2$ and $V_i \cup \{x,y\}$ induces a clique.  Any Berge cycle in $H_1$ has length at most $ 2(k-2) + 2 = 2k-2$.
\end{construction}

\begin{construction}Let $4 \leq r+1 \leq k \leq n/2$. Let $H_2 = H_2(k)$ be the $r$-graph with $V(H_2) = X \cup Y$ where $|X| = k-1$, $|Y| = n-(k-1)$, and $E(H_2)$ is the set of all hyperedges containing at most one vertex in $Y$. No Berge cycle can contain consecutive vertices in $Y$, so any Berge cycle has length at most $2k-2$. 
\end{construction}

Observe that both $H_1$ and $H_2$ have minimum degree ${k-1 \choose r-1}$. Moreover, $H_2$ is $(k-1)$-connected and can be defined for all $n\geq k$. Therefore, the bound in Theorem~\ref{mainthm} cannot be further decreased by requiring higher connectivity.

\medskip
{\bf Remark 1}.
Problems on  conditions for the existence of long Berge paths and cycles (in particular, Tur\'an-type analogs of the Erd\H os-Gallai Theorem) attracted recently considerable attention, see e.g.~\cite{EGMSTZ,FKL,GKL, GLSZ,GSZ,FKL2conn} and references in them. These results yield some Dirac-type bounds, but the implied bounds are significantly weaker than
the bound in Theorem~\ref{mainthm}.

\medskip
{\bf Remark 2}. The extremal hypergraph of the Tur\'an-type problem in~\cite{FKL2conn} (the maximum number of edges among all $r$-uniform, $2$-connected hypergraphs with no Berge cycle of length at least $\min\{2k, n\}$) contains $H_2$ as a subhypergraph. But the extremal hypergraph $H_1$ has fewer edges.

We also present a bound for $k=2$ that is better than given by Corollary~\ref{jackcor}:

\begin{prop}\label{k=2!} Let $3\leq r<n$ be positive integers. Then every  $n$-vertex $2$-connected $r$-graph  $H$ contains a Berge cycle of length at least $\min\{4,|E(H)|\}$. 
\end{prop}

A sharpness example is an $r$-graph $H_3=H_3(r,s)$ with vertex set $\{v_1,v_2\}\cup \bigcup_{i=1}^s U_i$ where $|U_i|=r-1$ for $1\leq i\leq s$ and edge set is $\bigcup_{i=1}^s\{e_{i,1},e_{i,2}\}$ where $e_{i,j}=U_i\cup\{v_j\}$ for 
$1\leq i\leq s$ and $1\leq j\leq 2$. This $r$-graph is $2$-connected for $s\geq 2$ and has no Berge cycles of length more than $4$.

\medskip
 A related notion is the {\em codiameter} of a hypergraph $H$ which is
 the maximum integer $k$ such that for every two vertices $u,v \in V(H)$, $H$ contains a Berge $u,v$-path of length at least $k$. (Recall that the length of a Berge path is the number of its edges.)

In graphs, having codiameter $k$ is equivalent to the property that for any two vertices $x,y$, graph $G + xy$ has a cycle of length at least $k+1$ passing through edge $xy$. This property is well studied, see~\cite{Fan1,Fan2,enom}.
  It was proved recently in~\cite{KLM3} that the bound $\delta(H) \geq {\lfloor n/2 \rfloor \choose r-1} + 1$ guarantees the largest possible codiameter, $n-1$.

As an application of our main theorem, we prove the following Dirac-type bound.

\begin{coro}\label{kpath} Let $n, k, r$ be positive integers with $n/2 \geq k \geq r+2$ and $r \geq 3$. If $H$ is an $r$-uniform, $n$-vertex, $2$-connected hypergraph with
\[\delta(H) \geq {k-1 \choose r-1} + 1,\]
then the codiameter of $H$ is at least $k$.
\end{coro}

For $n = q(k-2) + 2$, the construction $H_1(k)$ shows that Corollary~\ref{kpath} is sharp: the longest Berge path from $x$ to $y$ contains $k-1$ edges. We also note that $2$-connectivity is necessary: for large $n$ divisible by $r$, we may take $r$ copies of $K_{n/r}^{(r)}$ and a single edge intersecting each clique in one vertex. This hypergraph has minimum degree ${n/r - 1 \choose r-1}$ (which does not depend on $k$) but there are pairs of vertices that are connected  only by a one-edge  Berge path.

\subsection{Outline of the paper}

The structure of the paper is as follows. In Section~\ref{shortproofs} we present a simple proof of Proposition~\ref{k=2!} and derive Corollary~\ref{kpath} from Theorem~\ref{mainthm}.  In Section~\ref{setup!} we set up the proof of our main result,
Theorem~\ref{mainthm}. We introduce notation and define so called {\em lollipops}. Each lollipop is roughly speaking  a pair of a Berge cycle $C$ and a Berge path (or a partial Berge path, defined in the next section) $P$ such that $P$ starts in $C$ and extends outward. In particular, we define criteria for which we will choose an {\em optimal} lollipop $(C,P)$.

In the subsequent five sections we consider all possible cases of best lollipops $(C,P)$ and find a contradiction in each of them. 
In particular, in Section 4, we show that in an optimal lollipop, $P$ has a positive length. In Section 5, inspired by Dirac's proof of Theorem~\ref{dirac2}, we show that the end vertex of the $P$ cannot have too many neighbors in $P$. One of the key ingredients of the proof is a modification of a Dirac's lemma on paths in $2$-connected graphs (Lemma~\ref{diraclemma}).
In Sections 6 and 7, we show that $P$ must be a Berge path and cannot be too long. Finally in Section 8, using the structure of $(C,P)$ established in previous sections, we analyze how the neighborhoods of two vertices in $P$ can interact and conclude that we must be able to construct a longer cycle than $C$.

We note that if $k \geq n/2$ then by Theorem~\ref{oldmain}, $\delta(H) \geq {k-1 \choose r-1}+1 \geq {\lfloor (n-1)/2 \rfloor \choose r-1} + 1$ implies that $H$ contains a Berge cycle of length $n$. Thus when proving Theorem~\ref{mainthm} we will assume 
\begin{equation}\label{kn/2}
\mbox{\em $k < n/2$ and $\min\{2k, n\} = 2k$. }
\end{equation}

\section{Short proofs}\label{shortproofs}

In this section, we present a proof of Proposition~\ref{k=2!} and show how to derive Corollary~\ref{kpath}
from our main result.

\subsection{ Proof of Proposition~\ref{k=2!}}

\begin{proof}
  Suppose $H$ is a counter-example to the proposition, i.e. for some $3\leq r<n$, $H$ is $n$-vertex $2$-connected $r$-graph and each Berge cycle in $H$ has length at most $\min\{3,|E(H)|-1\}$. Since $H$   is 2-connected, by Corollary~\ref{jackcor}, $\delta(H)\leq 2$, so $\delta(H)= 2$.
  
 Assume first, a pair $\{u_1,v_1\}$ of vertices in $H$ is not in any edge. Since the incidence graph $I_H$ of $H$   is 2-connected, by Menger's Theorem it has a cycle
 $C=u_1,e_1,u_2,e_2,\ldots,u_s,e_s,v_1,f_1,v_2,\ldots,v_t,f_t,u_1$
   containing $u_1$ and $v_1$. Since no edge contains both $u_1$ and $v_1$, the four edges $e_1,e_s,f_1,f_t$ of $H$ are distinct. Then $C$ 
   corresponds to a cycle in $H$ of length at least $4$, a contradiction. Thus,
 \begin{equation}\label{nopair}\mbox{\em for each pair $\{u,v\}\subset V(H)$ there is an edge $e_{uv}$ containing $u$ and $v$.} \end{equation}
 
 Since $\delta(H) = 2$, let $u\in V(H)$ with $d_H(u)=2$ and $e_1,e_2$ be the edges containing $u$.  Let $A_0=e_1\cap e_2$, $A_1=e_2\setminus e_1$ and $A_2=e_1\setminus e_2$.
 
 By~\eqref{nopair}, $e_1\cup e_2=V(H)$. We claim that
 \begin{equation}\label{nopair2}\mbox{\em some edge $e_0$ of $H$ contains $A_1\cup A_2$.} \end{equation}
 Indeed, let $x_1\in A_2$ and $x_2\in A_1$. By~\eqref{nopair}, there is an edge $e_3$ containing $x_1$ and $x_2$.
 If $e_3$ omits some $y_1\in A_2$ and some $y_2\in A_1$, then again by~\eqref{nopair}, there is an edge $e_4$ containing $y_1$ and $y_2$, and so $H$ has $4$-cycle $x_1,e_1,y_1,e_4,y_2,e_2,x_2,e_3,x_1$, a contradiction.
 Thus we may assume $e_3\supset A_1$ and $y_1\in A_2\setminus e_3$. Since $|A_2|=|A_1|$, there is $y_2\in A_1-x_2$.
 Again by~\eqref{nopair}, there is an edge $e_4$ containing $y_1$ and $y_2$, and so $H$ has $4$-cycle $x_1,e_1,y_1,e_4,y_2,e_2,x_2,e_3,x_1$. This proves~\eqref{nopair2}.
 
 So, for $0\leq i\leq 2$, $e_i\supseteq V(H)\setminus A_i$. Since $|E(H)|\geq 4$, there is an edge $g\notin \{e_0,e_1,e_2\}$. If some two vertices of $g$ are in the same $A_i$, say $u,v\in g\cap A_0$, then $H$  has $4$-cycle $u,g,v,e_1,x_1,e_0,x_2,e_2,u$, where $x_1\in A_2$ and $x_2\in A_1$. Otherwise, $r=3$ and $g$ has a vertex in each of $A_0,A_1,A_2$.
 Since $|V(H)|\geq 4$, some $A_i$ has at least two vertices, say $|A_1|\geq 2$. For $0\leq i\leq 2$, let $u_i\in g\cap A_i$.
 Let $v\in A_1-u_1$. Then $H$  has $4$-cycle $u_0,g,u_1,e_2,v,e_0,u_2,e_1,u_0$. This contradiction finishes the proof.
\end{proof}

\subsection{Proof of Corollary~\ref{kpath} on  codiameters}

\begin{proof}
Suppose 
\begin{equation}\label{kr}
 n/2 \geq k \geq r+2 \geq 5,
 \end{equation}
and $H$ is an $r$-uniform, $n$-vertex, $2$-connected hypergraph with $\delta(H) \geq {k-1 \choose r-1} + 1$. Then by Theorem~\ref{mainthm}, $H$ contains a cycle $C= v_1, e_1, \ldots, v_c, e_c, v_1$ with $c \geq 2k$. 
Fix $u, v \in V(H)$. If $u,v \in V(C)$ then there  exists a segment in $C$ from $u$ to $v$ with at least $\lceil (c+2)/2 \rceil \geq k+1$ vertices. This is a path of length at least $k$. 

Otherwise, consider the incidence graph $I_H$ which is 2-connected. There exist shortest disjoint (graph) paths $P_1$ and $P_2$ in $I_H$ from $V(C) \cup E(C)$ to $\{u, v\}$, say $u \in P_1, v\in P_2$. If $u \in V(C)$, then we have $P_1 = u$ and similar for $v$. In $H$, $P_1$ and $P_2$ correspond to either Berge paths or partial Berge paths that end with $u$ and $v$ respectively. Let $a_1, a_2$ be the first elements of $P_1$ and $P_2$ respectively, and let $Q$ be the longer of the two $a_1, a_2$-segments along $C$. If the two  $a_1, a_2$-segments along $C$ have equal length, we choose one arbitrarily.

If without loss of generality, $u \in V(C)$, then $|V(Q)| \geq \lceil (c+1)/2\rceil \geq k+1$. Appending $P_2$ to the end of $Q$ gives a path of length at least $k+1$ from $u$ to $v$.
Finally,  
 if $u,v \notin  V(C)$, then $|V(Q)| \geq \lceil c/2 \rceil \geq k$, so $P_1 \cup Q \cup P_2$ is a $u,v$-path with at least $k+2$ vertices.
\end{proof}

\section{Setup and simple properties of best lollipops}\label{setup!}

In this section we present some hypergraph notation and define lollipops. We also derive a series of useful properties of optimal lollipops.

\subsection{Notation and setup}

For a hypergraph $H$, and a vertex $v \in V(H)$, \[N_{H}(v) = \{u \in V(H): \text{there exists } e \in E(H) \text{ such that } \{u,v\} \subset e\}\] is the {\em $H$-neighborhood} of $v$. 
The {\em closed $H$-neighborhood} of $v$ is the set $N_H[v]=N_H(v) \cup \{v\}$.
%

When $G$ is a subhypergraph of a hypergraph $H$ and $u,v\in V(H)$,
we say that $u$ and $v$ are {\em $G$-neighbors} if there exists an edge $e \in E(G)$ containing both $u$ and $v$.

 When we speak of an $x,y$-(Berge) path $P$ and
$a,b\in V(P)$, then  $P[a,b]$ denotes the unique segment of $P$ from $a$ to $b$.

Let $r\geq 3$. We consider a counter-example $H$. Taking into account~\eqref{kn/2}, $k<n/2$ and
 $H$ is a 2-connected $n$-vertex $r$-uniform hypergraph satisfying~\eqref{main} such that
 \begin{equation}\label{2kn}
 \mbox{\em $H$ does not contain a Berge cycle of length at least    $2k$.}
 \end{equation}

A {\em lollipop $(C,P)$} is a pair where $C$ is a Berge cycle and $P$ is a Berge path or a partial Berge path that satisfies one of the following:

 \begin{enumerate}
 \item[---] $P$ is a Berge path starting with a vertex in $C$, $|V(C) \cap V(P)| = 1$, and $|E(C) \cap E(P)| = 0$. We call such a pair $(C,P)$ an {\bf ordinary lollipop} (or o-lollipop for short). See Fig.~\ref{pics} (left).
 
 \item[---] $P$ is a partial Berge path starting with an edge in $C$, $|V(C) \cap V(P)| = 0$, and $|E(C) \cap E(P)| = 1$. We call such a pair $(C,P)$ a {\bf partial lollipop} (or p-lollipop for short). See Fig.~\ref{pics} (right).
 \end{enumerate}
 \begin{figure}
 \centering
 \includegraphics[scale=.95]{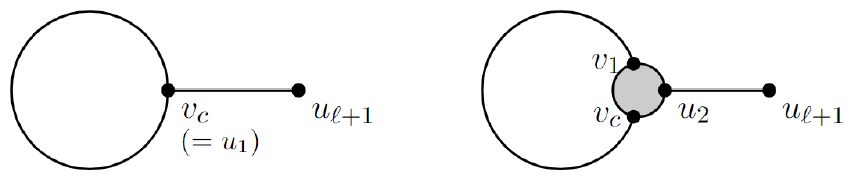}
 \caption{An $o$-lollipop and a $p$-lollipop.}
 \label{pics}
 \end{figure}
 
 A lollipop  $(C,P)$  {\em is better} than a lollipop   $(C',P')$  if
  \begin{enumerate}
\item[(a)]  $|V(C)|>|V(C')|$, or 
\item[(b)] Rule (a) does not distinguish $(C,P)$  from   $(C',P')$, and $|E(P)|>|E(P')|$; or

\item[(c)] Rules (a) and (b) do not distinguish $(C,P)$  from   $(C',P')$, and the total number of vertices of $V(P)-V(C)$ contained in the  edges of $C$ counted with multiplicities is larger than the total number of vertices of $V(P')-V(C')$ contained in the  edges of $C'$; or 
\item[(d)] Rules (a)--(c) do not distinguish $(C,P)$ from $(C',P')$, and $(C,P)$ is an o-lollipop while $(C',P')$ is a p-lollipop; or
\item[(e)] Rules (a)--(d) do not distinguish $(C,P)$  from   $(C',P')$, and
 the number of edges in $E(P)-E(C)$ fully contained in $V(P)-V(C)$ is larger than the number of edges in $E(P')-E(C')$ fully contained in $V(P')-V(C')$. \end{enumerate}

The criteria (a)--(e) define a partial ordering on the (finite) set of lollipops, and hence we can choose a best lollipop $(C,P)$. Say $C = v_1, e_1, \ldots, v_c, e_c, v_1$. If $(C,P)$ is a o-lollipop then let $P = u_1, f_1, \ldots, f_\ell, u_{\ell+1}$, where $u_1 = v_c$. If $(C,P)$ is a p-lollipop then let $P = f_1, u_2, f_2, \ldots, f_\ell, u_{\ell+1}$ where $f_1 = e_c$. With this notation, we have $|E(P)| = \ell$, $|V(P)| = \ell+1$ if $P$ is a Berge path, and $|V(P)|=\ell$ if $P$ is a partial Berge path. Assume $c <2k = \min\{2k,n\}$.

Denote by $H'$ the subhypergraph of $H$ with $V(H') = V(H)$ and $E(H') = E(H) - E(C) - E(P)$. Define 
\begin{equation}\label{H''}
H''=\left\{\begin{array}{ll} H'&\mbox{ when $(C,P)$ is a p-lollipop},\\
H'\cup \{f_1\}&\mbox{ when $(C,P)$ is a  o-lollipop.}
\end{array}\right.
\end{equation}

Since we consider mostly Berge paths and cycles, from now on we will refer to them simply as paths and cycles. We will differentiate graph paths and cycles when needed.

\subsection{Simple properties of best lollipops}
In this subsection we consider best lollipops $(C,P)$  and prove some basic claims to be used throughout the rest of the paper.
The following claim  immediately follows from the assumption~\eqref{2kn}  and $c<2k$.
\begin{claim}\label{longQ} 
\begin{enumerate}
\item[(a)] If $a_1 = e_i$ and $b_1 = e_j$ for some $i, j  \in [c]$, then the longer of the two subpaths of $C$ connecting 
$\{v_{i},v_{i+1}\}$ with $\{v_{j}, v_{j+1}\}$ and using neither of
 $e_i$ and  $e_j$  has at least $\lceil c/2 \rceil $ vertices. In particular, this path omits at most $k-1$ vertices in $C$. 

\item[(b)] If $a_1 = e_i$ and $b_1 = v_j$ for some $i, j  \in [c]$, then the longer of the two subpaths of $C$ connecting 
$\{v_{i},v_{i+1}\}$ with $v_{j}$ and not using 
 $e_i$  has at least $\lceil (c+1)/2 \rceil $ vertices. In particular, this path omits at most $k-1$ vertices in $C$.

\item[(c)] If $a_1 = v_i$ and $b_1 = v_j$ for some $i, j  \in [c]$, then the longer of the two subpaths of $C$ connecting 
$v_{i}$ with $v_{j}$   has at least $\lceil (c+2)/2 \rceil $ vertices. In particular, this path omits at most $k-2$ vertices in $C$. 
\end{enumerate}
\end{claim}

We call a path satisfying Claim~\ref{longQ} a {\em long $a_1, b_1$-segment of $C$}.

\begin{claim}\label{bigsmallcycle} Let $(C,P)$ be a best lollipop. For each $1\leq i\leq c$ and $2\leq m\leq \ell+1$, if some edge $g\notin E(C)$ contains $\{u_{m},v_i\}$, then\\
(a) neither $e_{i-1}$ nor $e_i$ intersect $V(P)-u_1$, and\\
(b) no edge in $H''$ intersects both  $V(P)-u_1$ and 
$\{v_{i-1},v_{i+1}\}$ (indices count modulo $c$).

In particular, the set $N_{H''}(V(P)-u_1)\cap V(C)$ does not contain two consecutive vertices of $C$.
\end{claim}
\begin{proof}Let $g \notin E(C)$ contain $\{u_m, v_i\}$ such that if $g \in E(P)$, say $g = f_q$, then we may assume $u_m = u_{q+1}$. 
Suppose $ e_{i-1}$ contains $u_j$ for some $2\leq j\leq \ell+1$. If either $j \geq m$ or $g \neq f_{m-1}$, then we may replace the segment $v_{i-1}, e_{i-1}, v_i$ in $C$ with $v_{i-1}, e_{i-1}, u_{j},P[u_j,u_m],u_m, g, v_{i}$. Otherwise we replace the segment with $v_{i-1}, e_{i-1}, u_j, P[u_j, u_{m-1}], u_{m-1}, g, v_i$. We obtain a longer cycle, contradicting the choice of $C$. The case with $u_{j} \in e_i$ is symmetric. This proves (a).

Suppose now some $e\in E(H'')$ contains $\{u_{j},v_{i-1}\}$ for some $2\leq j\leq \ell+1$ (the case when $e\supset \{u_{j},v_{i+1}\}$ is symmetric).
If $e\neq g$, then similarly to before we may replace the segment $v_{i-1}, e_{i-1}, v_i$ in $C$ with $v_{i-1}, e, u_{j},P[u_j,u_m],u_m, g, v_{i}$ or $v_{i-1}, e, u_j, P[u_j, u_{m-1}], u_{m-1}, g$
 to get a longer cycle.

If $e = g$, then by (a), $ e_{i-1}\cap (V(P)-u_1)=\emptyset$. Note that in this case $g \in E(H'')$. Let $C'$ be obtained from $C$ by
 replacing the edge $e_{i-1}$ with $g$. If $g\neq f_1$, then we let $P'=P$, otherwise, by the definition~\eqref{H''} of $H''$,
 $P$ is a path, and we define partial path $P'=f_1,u_2,f_2,\ldots,u_{\ell+1}$.
 Then  $(C',P')$ is better than  $(C,P)$ by Rule (c) in the definition of better lollipops. 
\end{proof}

%

Call a lollipop $(C',P')$ {\em good} if $|E(C')|=c$ and $|E(P')|=\ell$. In particular, each best lollipop is a good lollipop.

\begin{claim}\label{allneighbors}
Suppose $(C,P)$ is a good lollipop.
 Let $\wt{H}$ be the subhypergraph of $H$ with $E(\wt{H}) = E(H) - E(C) - E(P)$. 
 Then all $\wt{H}$-neighbors of $u_{\ell+1}$ are in $V(C) \cup V(P)$, and moreover

(1) if $(C,P)$ is an $o$-lollipop, then $u_{\ell + 1}$ has no $\wt{H}$-neighbors in $\{v_1, v_2, \ldots, v_{\ell}\} \cup \{v_{c-1}, v_{c-2}, \ldots, v_{c-\ell}\}$, and $u_{\ell + 1}$ is not in any edge in the set $\{e_1, e_2, \ldots, e_{\ell-1}\} \cup \{e_c, e_{c-1}, \ldots, e_{c-\ell}\}$, 

(2) if $(C,P)$ is a $p$-lollipop, then $u_{\ell+1}$ has no $\wt{H}$-neighbors in $\{v_1, v_2, \ldots, v_\ell\} \cup \{v_c, v_{c-1}, \ldots, v_{c-\ell+1}\}$, and $u_{\ell+1}$ is not in any edge in the set $\{e_1, \ldots, e_{\ell-1}\} \cup \{e_{c-1}, \ldots, e_{c-(\ell-1)}\}$.
\end{claim}

\begin{proof}Let $e \in E(\wt{H})$ contain $u_{\ell+1}$. Suppose first there is a vertex $y \in V(H) - (V(C) \cup V(P))$ such that $y \in e$. Let $P'$ be the path  obtained 
from $P$ by adding  edge $e$ and vertex $y$ to the end of $P$. Then $(C, P')$ is a lollipop with $|V(P')| > |V(P)|$, a contradiction.

Now suppose $e$ contains $v_i$ for some $i \in \{1, \ldots,\ell\}$. Then we can replace the segment $v_c, e_c, v_1, \ldots, v_i$ from $v_c$ to $v_i$ in $C$ with the path $v_c, e_c, P[e_c, u_{\ell+1}], u_{\ell+1}, e, v_i$ to obtain a cycle of length at least $c - (\ell-1) + \ell > c$, contradicting the choice of $C$. 

The proof for $i \in \{c, \ldots, c-\ell\}$ or $i \in\{c, \ldots, c-\ell+1\}$ is very similar, but when $(C,P)$ is a $p$-lollipop, we replace the segment $v_1, e_c, v_c, \ldots, e_{i}, v_i$ instead with $v_1, e_c, P[e_c, u_{\ell+1}], u_{\ell+1}, e, v_i$.

Finally suppose $(C,P)$ is an $o$-lollipop and for some $1\leq i \leq \ell$, $u_{\ell + 1} \in e_{i-1}$ (modulo $c$). The cycle obtained by replacing the segment from $v_c$ to $v_{i}$ with the path $v_c, P,u_{\ell+1}, e_{i-1}, v_{i}$ has length at least $c + 1$, contradicting the choice of $C$. The argument for $e_{c - i }$ and the argument in the case $(C,P)$ is a $p$-lollipop and $e_{i-1} \neq e_c$ are similar.
\end{proof}





\begin{claim}\label{bestp} Let $(C,P)$ be  a best lollipop.

(A) If $u_{\ell+1}\in f_m$ for some $1\leq m\leq \ell-1$ and $P'$ is obtained from $P$ by replacing the subpath 
$u_m,f_m,u_{m+1},\ldots,u_{\ell+1}$ with the subpath $u_m,f_m,u_{\ell+1},f_\ell,u_\ell,\ldots,u_{m+1}$, then $(C,P')$ also is a best lollipop.

(B) If some edge $g\in E(H')$ contains $V(P)-V(C)$ or is contained in $V(P)-V(C)$  and contains $\{u_{\ell+1},u_m\}$  for some $1\leq m\leq \ell-1$, and if $P'$ is obtained from $P$ by replacing the subpath 
$u_m,f_m,u_{m+1},\ldots,u_{\ell+1}$ with the subpath $u_m,g,u_{\ell+1},f_\ell,u_\ell,\ldots,u_{m+1}$, then $(C,P')$ also is a best lollipop.
\end{claim}

\begin{proof} Let us check the definition of a best lollipop.
 Part (A) holds because  the vertex set and edge set of $P'-V(C)$ are the same as those of $P-V(C)$.
 
 In Part (B),  $V(P')-V(C)=V(P)-V(C)$, and $E(P')$ is obtained from  $E(P)$ by deleting $f_m$ and adding $g$. But since $g$
 contains $V(P)-V(C)$ or is contained in $V(P)-V(C)$, $(C,P)$ cannot be better than $(C,P')$.
\end{proof}

\begin{claim}\label{shortpath}For $2\leq q \leq \ell$ and $1 \leq i,j \leq c$, the following hold:

(1) If $u_q \in e_i$ and $u_{\ell+1} \in e_j$ then $j=i$ or $|j-i| \geq (\ell+1) - q + 1$.

(2) If there exists an edge $e \in E(H'')$ such that $\{v_i, u_q\} \subset e$, and if $u_{\ell+1} \in e_j$, then either $j > i$ and $j-i \geq (\ell+1) - q+1$, or $i > j$ and $i-j \geq (\ell+1)-q+2$.

(3) If there exist distinct edges $e, f \in E(H'')$ such that $\{v_i, u_q\} \subset e$ and $\{v_j, u_{\ell+1}\} \subset f$, then $j=i$ or $|j-i| \geq (\ell+1) - q + 2$.
\end{claim}

\begin{proof}We will prove (1). If $j\neq i$, then we can replace the segment of $C$ from $e_i$ to $e_j$ containing $|j-i|$ vertices with $e_j, u_{\ell+1}, P[u_{\ell+1}, u_q], u_q, e_i$ which contains $(\ell+1)-q + 1$ vertices. The new cycle  cannot be longer than $C$. The proofs for (2) and (3) are similar so we omit them.
\end{proof}


\section{Nontrivial paths in best lollipops}
\label{secl1}

In this section, we show that the path or partial path $P$ has length at least $2$. In particular, since $H$ is connected, and $|C|<n$, there is an edge intersecting both $V(C)$ and $V(H)-V(C)$. Thus $\ell\geq 1$.
Below we show in fact $\ell \geq 2$ using the notion of {\em expanding} sets that can be used to modify $C$ into a longer cycle.

Suppose $\ell=1$ and $u_2$ is the unique vertex in $V(P)\setminus V(C)$. Say that a set $W\subseteq V(C)$ is {\em $u_2$-expanding} 
if for every distinct $v_{j},v_{j'}\in W$, there is a $v_{j},v_{j'}$-path $Q(v_{j},v_{j'})$
whose all internal vertices are not in $V(C)\cup \{u_2\}$ and all edges are in $E(H)\setminus E(C)$. One example of a  $u_2$-expanding set is $V(C)\cap g$ where $g$ is any edge in $E(H)\setminus E(C)$. Another useful example is a set of the form $N_{H'}(w)\cap V(C)$ for a vertex $w\in V(H)-V(C)-u_2$.

Suppose $W$ is a   $u_2$-expanding set and 
$v_j, v_{j'} \in W$ where $j<j'$. If $u_2 \in e_j\cap e_{j'}$, then the cycle \[ v_1, e_1, v_2, \ldots, e_{j-1}, v_j, Q(v_{j},v_{j'}), v_{j'}, e_{j'-1}, v_{j'-1}, \ldots,  v_{j+1}, e_j, u_2, e_{j'}, v_{j'+1}, e_{j'+1}, \ldots, e_{c-1}, v_c, e_c, v_1
\]
is longer than $C$, a contradiction. A symmetric longer cycle can be found if $u_2 \in e_{j-1}\cap e_{j'-1}$. Thus 
\begin{equation}\label{no2}
\mbox{\em
$u_2$ is contained in at most one edge of $\{e_j: v_j \in W\}$ and in at most one edge of $\{e_{j-1}: v_j\in W\}$. }
\end{equation}
Therefore,
\begin{equation}\label{l1pl}
\parbox{14cm}{\em if the vertices of $W$ form on $C$ exactly $q$ intervals of consecutive vertices and $B$ is the set of edges in $C$ containing $u_2$, then   $|B|\leq c-|W|+1-q+1$. Moreover, if $q=1$, say $W=\{v_{j_1},v_{j_1+1},\ldots,v_{j_1+|W|-1}\}$ and $|B|= c-|W|+1$, then $B=E(C)\setminus \{e_{j_1},e_{j_1+1},\ldots,e_{j_1+|W|-2}\}$.}
\end{equation}


Now we are ready to prove that $\ell \geq 2$. 

\begin{lemma}\label{ell3}
Suppose   $n/2\geq k\geq r+2\geq 5$ and  $H$ is a $2$-connected $r$-graph satisfying~\eqref{main}.
Let $(C,P)$ be a best lollipop. If $c < 2k$, 
then $\ell=|E(P)|\geq 2$.
\end{lemma}
\begin{proof} Suppose $\ell=1$.
If there exists $ e \in E(H) - E(C)$ containing at least 2 vertices $u, u' \notin V(C)$, then let $P'$ be a shortest path or partial path from $V(C) \cup E(C)$ to $\{u,u'\}$ which avoids $e$. Such a $P'$ exists because $H$ is 2-connected. Without loss of generality, $P'$ ends with $u$. Then $(C, (P',e,u'))$ is  better  than $(C,P)$. It follows that
\begin{equation}\label{1c}
\hbox{for each $e \in E(H) - E(C)$, $|e \cap V(C)| \geq r-1$.}
\end{equation}

{\bf Case 1:} $P$ is a path, say $P=v_c,f_1,u_2$. Recall that $H''=H'\cup \{f_1\}$.
 Since  $d_H(u_2)\geq 2$ and every edge containing $u_2$ intersects $C$,
by the maximality of $\ell$,  $N_{H''}(u_2) \subseteq V(C)$.
 Let $A=N_{H''}(u_2)$, $a=|A|$,
 $B=\{e_i\in E(C): u_2\in e_i\}$  and $b=|B|$. 
By Claim~\ref{bigsmallcycle}, $A$ does not intersect the set $\bigcup_{e_i \in B} \{v_i, v_{i+1}\}$ and
no two vertices of $A$ are consecutive on $C$.
 Therefore,
 \begin{equation}\label{vcr} 
 2k-1\geq c\geq 2a+b.
\end{equation}  
 It follows that
 \begin{equation}\label{abr} 
1+{k-1\choose r-1}\leq d_H(u_2)\leq {a\choose r-1}+b\leq {a\choose r-1}+c-2a.
\end{equation}

 \medskip
{\bf Case 1.1:}  $d_{H''}(u_2)\geq 2$. Then   $a\geq r$. 
Since the RHS of~\eqref{abr} is monotonically increasing with $a$ when $a\geq r-1\geq 2$, if $a\leq k-2$, then
$1+{k-1\choose r-1}\leq {k-2\choose r-1}+c-2k+4$ and hence
$${k-2\choose r-2}\leq c-(2k-3)\leq 2,$$
a contradiction.

Suppose now $a=k-1$. Then~\eqref{abr} yields $b\geq 1$,~\eqref{vcr} yields $b\leq 1$, and in order to have equality, $c=2k-1$ and all $r$-tuples of vertices containing $u_2$ and contained in $A\cup \{u_2\}$ are edges of $H''$. 

It is convenient in this case to rename the vertices in  $C$ so that $B=\{ e_{2k-1}\}$ and $A=\{v_2,v_4,\ldots,v_{2k-2}\}$. Since $k=a+1\geq r+1$, for every $1\leq j\leq k-1$ we can choose an edge $g_{2j}\in E(H'')$ containing $u_2$ and $v_{2j}$ so that all $g_{2j}$ are distinct. 
Since $r\leq k-2$, some vertex in $V(C)-A$ is not in $e_{2k-1}$, say $v_{2i+1}\notin e_{2k-1}$ for some $1\leq i\leq k-2$. 

Suppose $v_{2i+1}\in e_j$ for some $ j\in [2k-2]-\{{2i},{2i+1}\}$. By symmetry, we may assume  $j> 2i+1$. If $j$ is even then the cycle 
$$C_j=v_{2i+1},e_j,v_{j+1},e_{j+1},v_{j+2},\ldots,v_{2i},g_{2i},u_2,g_{j},v_{j},e_{j-1},v_{j-1},\ldots,v_{2i+1}$$
is longer than $C$. If $j$ is odd then by the choice of $v_{2i+1}$, $j\neq 2k-1$, and the cycle 
$$C'_j=v_{2i+1},e_j,v_{j},e_{j-1},v_{j-1},\ldots,v_{2i+2},g_{2i+2},u_2,g_{j+1},v_{j+1},e_{j+1},v_{j+2},\ldots,v_{2i+1}$$
is longer than $C$, a contradiction. 

Similarly, if for some odd $j\neq 2i+1$ there is an edge $h_j\in E(H'')$ containing $v_{2i+1}$ and $v_j$, then we may assume $j>2i+1$, and the cycle $C''_j$ obtained from $C'_j$ by replacing $e_j$ with $h_j$ is longer than $C$. Recalling that $A$ is the set of vertices with even indices in $C$ we obtain $N_{H-\{e_{2i},e_{2i+1}\}}(v_{2i+1})\cap V(C)\subseteq A$. Since $|A|=k-1$ and $d_H(v_{2i+1})\geq {k-1\choose r-1}+1$, some edge $h\in E(H'')\cup \{e_{2i},e_{2i+1}\}$ containing $v_{2i+1}$ contains also a vertex $w\notin V(C)$.
Since $v_{2i+1}\notin A\cup e_{2k-1}$, $w\neq u_2$.
 Consider the lollipop
$(C_1,P_1)$ where $C_1$ is obtained from $C$ by replacing the subpath $v_{2i},e_{2i},v_{2i+1},e_{2i+1},v_{2i+2}$ with
the subpath $v_{2i},g_{2i},u_2,g_{2i+2},v_{2i+2}$, and $P_1=v_{2i},e_{2i},v_{2i+1},h,w$. This lollipop satisfies $|V(C_1)|=|V(C)|$ but $|E(P_1)| > |E(P)|$, contradicting the choice of $(C,P)$.

\medskip
{\bf Case 1.2:} $d_{H''}(u_2)=1$. Then $d_H(u_2)=1+b$, $A=f_1\cap V(C)$ and $a=r-1$. By~\eqref{vcr},
  $d_H(u_2)=1+b\leq 1+c-2a=c-2r+3$. In particular,  
  \begin{equation}\label{r+2}
1+{k-1\choose 2}\leq 1+{k-1\choose r-1}\leq d_H(u_2)\leq (2k-1)-2r+3\leq 2k-4,
\end{equation}
and thus $k^2-7k+12\leq 0$.  For $k\geq r+2\geq 5$, this is impossible.

\medskip
{\bf Case 2:} $P$ is a partial path, say
  $P=e_c,u_2$. If there is an edge $h\in E(H')$ containing $u_2$, then by~\eqref{1c}, $h$ contains some $v_j\in V(C)$.
So, the lollipop. $(C,P')$ where $P'=v_j,h,u_2$ also is better than $(C,P)$ by Rule (d), a contradiction. 
  So, 
  $d_H(u_2)=b$, where $B=\{e_i\in E(C): u_2\in e_i\}$  and $b=|B|$. 

Suppose there exists $w\in V(H)-(V(C)\cup \{u_2\})$. Let us show that
  \begin{equation}\label{d(w)}
d_{H'}(w)\leq 1. 
\end{equation}
  Indeed, suppose $g_1,g_2\in E(H')$ and $w\in g_1\cap g_2$. Let $W=V(C)\cap (g_1\cup g_2)$. As observed before, this $W$ is $u_2$-expanding. Since $g_2\neq g_1$, by~\eqref{1c}, $|W|\geq r$. Also, by Claim~\ref{bigsmallcycle},
  vertices in $g_2$ could not be next to  vertices in $g_1$ on $C$. Thus if $|W|= r$, then $|g_1\cap g_2|=r-1$, and hence no two vertices  of $W$ are consecutive on $C$. In this case, by~\eqref{l1pl},
 $b\leq c-|W|+1-q+1$ where $q=|W|=r$. So, similarly to~\eqref{r+2} we get
 $$ 1+{k-1\choose 2}\leq  d_H(u_2)\leq c-2r+2\leq 2k-5,$$
 which yields $k^2-7k+14\leq 0$, an impossibility. 
 Thus $|W|\geq r+1$. But still since  vertices in $g_2$ could not be next to  vertices in $g_1$ on $C$, $q\geq 2$.
So, we again get
 $$ 1+{k-1\choose 2}\leq  d_H(u_2)\leq c-(r+1)+1-2+1\leq 2k-5,$$  
 and come to a contradiction.
 This proves~\eqref{d(w)}.

 If $n=c+x$ and $|E(H)|=c+y$, then 
  \begin{equation}\label{c+x}
  k(c+x)=k\cdot n\leq \sum\nolimits_{v\in V(H)}d_H(v)=r(c+y).
  \end{equation}
   
\medskip
If $r\geq n/2$, then by Theorem~\ref{mainold2}, $H$ has a Hamiltonian cycle; thus $n>2r$. So, we conclude from~\eqref{c+x} that
$$y\geq \frac{(r+2)(c+x)}{r}-c= \frac{2(c+x)}{r}+x=\frac{2n}{r}+x>4+x.
$$
 Since $d_{H'}(u_2)=0$ and by~\eqref{d(w)}, at most $x-1$ edges in $H'$ contain a vertex outside of $V(C)$. It follows that  
   at least $6$ edges of $H'$ are contained in $V(C)$. If at least one of these edges is not an interval of consecutive vertices on $C$,~\eqref{l1pl} yields $b\leq c-r$. Also if all of these edges form intervals on $C$, then the ``Moreover" part of~\eqref{l1pl} yields $b\leq c-r$. Hence,
$$1+{k-1\choose 2}\leq d_H(u_2)=b\leq c-r\leq 2k-4,$$
and thus $k^2-7k+12\leq 0$.  For $k\geq r+2\geq 5$, this is impossible. 
  \end{proof}

\section{Vertex $u_{\ell+1}$ has few neighbors in $P$}

In this section we show that $u_{\ell+1}$ cannot have too many $H'$-neighbors in $P$ or be contained in too many edges in $P$. In particular, we will prove that it has at most $k-2$ such neighbors, and is in at most $k-1$ such edges.  The proof is a modification of Dirac's proof of Theorem~\ref{dirac2} for 2-connected graphs. The interested reader may look at the relevant sections in~\cite{D} or~\cite{KLM2} to see the proof idea in the simpler setting of graphs.

\bigskip

Let $S_1=(N_{H'}(u_{\ell+1})\cap V(P))\cup \{u_\ell\}$, $S_2=\{u_m\in V(P): u_{\ell+1}\in f_m \mbox{ and } u_{m}\notin S_1\}$ and
 $S=S_1\cup S_2$.

We will prove a series of claims. In each claim, we construct a cycle containing almost all of $S_1 \cup S_2$ and at least half the vertices in $C$. Thus if $S_1 \cup S_2$ or $S_1$ is too large (in particular, if $|S_1 \cup S_2| \geq k$ or $|S_1| \geq k-1$), we obtain a cycle that is longer than $C$. 

We use $h'(a,b)$ to denote an edge in $H'$ containing $a$ and $b$ if we know such an edge exists.
By definition, if $u_m\in S_2$ then there is no edge in $H'$ containing $u_m$ and $u_{\ell+1}$. In this case,
$h'(u_m,u_{\ell+1})$ denotes $f_m$.

If the smallest index $i$ with $u_i\in S_1\cup S_2$ is such that $u_i\in S_2$ then we denote this index by $i_1-1$, otherwise if $u_i \in S_1$ then we denote it by $i_1$. Let the other indices $i$ such that  $u_i\in S$ be $i_2,\ldots,i_{\alpha}$ in increasing order. 

Index the vertices of $S_1$ by $j_1, j_2, \ldots, j_\beta$ in increasing order.
 

If $(C,P)$ is an $o$-lollipop, then let $X=V(C)-v_c$, otherwise let $X = V(C)$. Set $Y=\{u_{i_1+1},u_{i_1+2},\ldots,u_{\ell+1}\}$ and $Z = \{u_{j_1+1}, u_{j_1+2}, \ldots, u_{\ell+1}\}$. Observe that $Z \subseteq Y$. 


\begin{claim}\label{xyh'2} If $|S_1 \cup S_2| \geq k-1$ then no edge in $H'$ intersects both $X$ and $Y$.
\end{claim}
\begin{proof} Suppose $(C,P)$ is a lollipop and an edge in $E(H')$  intersects both $X$ and $Y$. Among such edges, choose $e$ containing $u_i$ with the maximum possible $i$.  Let $i'$ be the largest index less than $i$ such that $u_{i'}\in S_1\cup S_2$. 

By the definition of $Y$, $i>i_1$ and hence $i'\geq i_1-1$. Suppose a vertex in $X\cap e$ is $v_j$. 
Let $Q$ be a long $v_c, v_j$-segment of $C$ guaranteed by Claim~\ref{longQ}. If $i=\ell+1$, then consider the cycle 
$$C_0=v_c,Q,v_j,e,u_{\ell+1},f_{\ell},u_{\ell},\ldots,f_1, v_c.$$
If $(C,P)$ is an o-lollipop, then $C_0$
has at least $c-(k-2)$ vertices in $C$ and at least $k$ vertices in $S_1 \cup S_2 \cup \{u_{\ell+1}\}$, at most one of which is in $C$ (namely $v_c = u_1$). So $|C_0| \geq c-(k-2) + (k-1) >c$,
a contradiction. If $(C,P)$ is a p-lollipop, then $C_0$
is guaranteed only $c-(k-1)$ vertices in $C$, but none of the at least $k$ vertices in $S_1 \cup S_2 \cup \{u_{\ell+1}\}$ is in $C$. So $|C_0| \geq c-(k-1) + k >c$, again.

Thus, suppose $i\leq \ell$.
Then $h'(u_{i'},u_{\ell+1})\neq e$ and hence
$$C'_0=v_c,Q,v_j,e,u_i,f_i,\ldots,u_{\ell},f_\ell,u_{\ell+1},h'(u_{\ell+1},u_{i'}),f_{i'-1},\ldots, f_1, v_c$$
is a cycle. Similarly to $C_0$, it has at least $c-(k-2)+(k-1)>c$  vertices when  $(C,P)$ is an o-lollipop, 
and at least $c-(k-1)+k$  vertices when  $(C,P)$ is a  p-lollipop,
 a contradiction.
\end{proof}

\begin{claim}\label{xyf3}If $|S_1 \cup S_2| \geq k$ then no $f_m$ with $m\geq i_1$ intersects  $X$, and if $|S_1| \geq k-1$, then no $f_m$ with $m \geq j_1$ intersects $X$.
\end{claim}

\begin{proof} Suppose $m\geq i_1$ and $f_m$ contains some  $v_j\in V(C)$. By symmetry, we may assume that $j\leq c/2$ when 
 $(C,P)$ is an o-lollipop and $j\leq (c+1)/2$ when 
 $(C,P)$ is a p-lollipop. 
  Let $Q$ be the path $v_c,e_{c-1},v_{c-1},\ldots,e_j, v_j$.

Suppose first that $|S_1 \cup S_2| \geq k$. If $m=i_1$ and the smallest index $i$ with $u_i \in S_1 \cup S_2$ is such that $u_i \in S_1$, then let $i'=m$; otherwise, let $i$ be the largest index less than $m$ such that $u_{i'} \in S_1 \cup S_2$. 
Then 
$$C''_0=v_c,Q,v_{j},f_m,u_{m+1},f_{m+1},\ldots,u_{\ell},f_\ell,u_{\ell+1},h'(u_{\ell+1},u_{i'}),f_{i'-1},\ldots f_1,v_c$$
is a cycle. It contains all vertices in $S_1\cup S_2\cup \{u_{\ell+1}\}$  apart from $u_m$. If $(S_1 \cup S_2)\cap V(C) \neq \emptyset$, then $(C,P)$ must be an o-lollipop and $u_1 \in S_1 \cup S_2$.
Hence $C''_0$ has  at least $c-(j-1)+(k+1)-2\geq (c-c/2)+k>c$  vertices, a contradiction. Otherwise, if $(S_1 \cup S_2) \cap V(C) = \emptyset$, then $C''_0$ has at least $c-(j-1) + (k+1) -1\geq (c-(c+1)/2)+k+1 > c$ vertices.

Now suppose $|S_1| \geq k-1$. In this case, let $i'$ be the largest index that is at most $m$ such that $u_{i'} \in S_1$. Then the same cycle $C''_0$ as above  contains all vertices in $S_1  \cup \{u_{\ell+1}\}$. Similarly, we get either $|C''_0| > c-(c/2-1)+k-1 > c$ or $|C''_0| > c-((c+1)/2-1) + k > c$. 
\end{proof}

\begin{claim}\label{xyc2} Suppose $|S_1 \cup S_2| \geq k$. If $(C,P)$ is an o-lollipop, then no  $e_j \in E(C)$ intersects  $Y$. If $(C,P)$ is a p-lollipop then no $e_j \in E(C)$ with $j \neq c$ intersects $Y$.\end{claim}

\begin{proof} Suppose   $e_j\in V(C)$ contains some $u_i\in Y$ where $j \neq c$ when $(C,P)$ is a p-lollipop. 
By symmetry, we may assume that $j\leq c/2$.   Let $Q$ be the path $v_c,e_{c-1},v_{c-1},\ldots,e_{j+1}, v_{j+1}$.

Let $i'$ be the largest index less than $i$ such that $u_{i'}\in S_1\cup S_2$. 
Consider the cycle
 \[C_1 = v_c, Q, v_{j+1}, e_j, u_i, P[u_i, u_{\ell+1}], u_{\ell+1}, h'(u_{\ell+1}, u_{i'}), u_{i'}, P[u_{i'}, v_c], v_c.\] 
 
It contains at least $c- c/2$ vertices in $C$ and all vertices in $S_1\cup S_2\cup \{u_{\ell+1}\}$.
Hence $|C_1|\geq  c/2+(k+1)-1>c$. 
\end{proof}

\begin{claim}\label{xyc2'} Suppose $|S_1| \geq k-1$ and some edge $e_j \in E(C)$ contains some $u_i \in Z$. Then either $(C,P)$ is an o-lollipop, $c=2k-1$, $e_j=e_{k-1}$,  $|S_1|=k-1$ and $j_1=1$, or $(C,P)$ is a p-lollipop and $j=c$.\end{claim}

\begin{proof} Suppose   $e_j\in V(C)$ contains some $u_i\in Z$ where $j \neq c$ when $(C,P)$ is a p-lollipop. 
As in the proof of Claim~\ref{xyc2}, we may assume that $j\leq c/2$, and if $c$ is an o-lollipop, we may assume $j \leq (c-1)/2$.   Let $Q$ be the path $v_c,e_{c-1},v_{c-1},\ldots,e_{j+1}, v_{j+1}$.

Let $i'$ be the largest index less than $i$ such that $u_{i'} \in S_1$. Consider
 \[C_1 = v_c, Q, v_{j+1}, e_j, u_i, P[u_i, u_{\ell+1}], u_{\ell+1}, h'(u_{\ell+1}, u_{i'}), u_i', P[u_i', v_c], v_c.\] 
 
 If   $(C,P)$ is a  o-lollipop, then     
 \begin{equation}\label{vc1}
 |V(C_1)|\geq c-j+k-1\geq  c-(c-1)/2 + k-1=c+\frac{2k-c-1}{2}\geq c
 \end{equation}
  with equality only if  $c=2k-1$, $j=k-1$, $u_1 \in S_1$, and $|S_1| = k-1$.

If $(C,P)$ is a p-lollipop, then $S_1\cap V(C)=\emptyset$, so instead of $|V(C_1)|\geq c-j+k-1$ as in~\eqref{vc1} we have
$|V(C_1)|\geq c-j+k$ and conclude that $|V(C_1)|>c$.
\end{proof}

Claims~\ref{xyh'2}--\ref{xyc2'} together can be summarized as the following two corollaries.

\begin{cor}\label{XY2} Suppose $|S_1 \cup S_2| \geq k$. Then the only edges in $H$ 
 that may intersect both $X$ and $Y$ are $f_1,\ldots,f_{i_1-1}$.
\end{cor}

\begin{cor}\label{XY2'} Suppose $|S_1| \geq k-1$ and an edge  $g\in E(H)$ intersects $X$ and $Z$. 
Then either $g\in \{f_1,\ldots,f_{j_1-1}\}$ or
$(C,P)$ is an o-lollipop,  $g=e_{k-1}$, $c=2k-1$, $|S_1| = k-1$, and $j_1 = 1$. \end{cor}

Finally we will show that $|S_1|$ and $|S_1 \cup S_2|$ cannot be too large. For this, we use the notion of {\em aligned} paths in graphs introduced in~\cite{D} and apply Lemma~\ref{diraclemma} below 
  to the incidence bigraph $I_H$ of $H$. 

Let $P$ and $P'$ be paths in a graph starting from the same vertex. We say $P'$ is {\em aligned with} $P$ if for all $u,v \in V(P) \cap V(P')$, if $u$ appears before $v$ in $P$ then $u$ also appears before $v$ in $P'$. 

\begin{lemma}[Lemma 5 in~\cite{KLM2}]\label{diraclemma}Let $P$ be an $x,y$-path in a 2-connected graph $G$, and let $z \in V(P)$. Then there exists an $x, z$-path $P_1$ and an $x,y$-path $P_2$ such that

 (a) $V(P_1) \cap V(P_2) = \{x\}$ $\quad$ and (b)
 each of $P_1$ and $P_2$ is aligned with $P$.
\end{lemma}

\begin{lemma}\label{BL2}  
(A) $|S_1\cup S_2|\leq k-1$, \; and\;
(B) $|S_1| \leq k-2$. 
\end{lemma}

\begin{proof}
 Recall that by Lemma~\ref{ell3}, $\ell \geq 2$. We first prove (A).
Suppose towards contradiction that $|S_1 \cup S_2| \geq k$.

{\bf Case 1}. The smallest index $i$ with $u_i \in S_1 \cup S_2$ satisfies $u_i \in S_2$. By the definition of $S_2$ and $i_1$, $i=i_1-1$, and $i$ is the unique index less than $i_1$ such that $u_{\ell+1}\in f_i$. In particular,
 $i_1 \geq 2$ and moreover if $(C,P)$ is a p-lollipop, since the first vertex of $V(P)$ is $u_2$, $i_1-1 \geq 2$. 

Consider the 2-connected incidence bipartite graph $I_H$ of $H$ and the (graph) path
  \[P' = v_1, e_1, v_2, \ldots, v_c, f_1, \ldots, f_{\ell}, u_{\ell+1}\] in $I_H$. 
  We apply Lemma~\ref{diraclemma} to $P'$ with $z = u_{i_1}$ to obtain two (graph) paths $P_1$ and $P_2$ satisfying (a) and (b) in $I_H$. 

We modify $P_i$ as follows: if $P_i = w_1, w_2, \ldots, w_{j_i}$, let $q_i$ be the last index such that $w_{q_i} \in X':= \{v_1, e_1, \ldots, v_c, e_c\}$ and let $p_i$ be the first index such that $w_{p_i} \in Y':=\{u_{i_1}, f_{i_1}, u_{i_1+1}, \ldots, f_{\ell}, u_{\ell+1}\}$. 

If $w_{p_i} =u_s$ for some $s$, then set $P_i'= P_i[w_{q_i}, w_{p_i}]$. If $w_{p_i} = f_s$ for some $s$, then set $P_i' = P_i[w_{q_i}, w_{p_i}], u_{s+1}$. 

Observe that $P_1'$ and $P_2'$ are  Berge paths or partial Berge paths in $H$. Moreover, $P_1'$ ends with vertex $z = u_{i_1}$
and contains no other elements of $Y'$ since it is aligned  with $P'$. It is possible that $f_{i_1-1}$ is in $P'_1$.

If both $P_1'$ and $P_2'$  begin with $v_1$,  then some $P_i'$ avoids $f_1$ and first intersects the set $\{u_2, f_2, \ldots, u_{\ell+1}\}$ in $I_H$ at some vertex $w_j$. Then replacing the segment $v_1, e_c, v_c$ in $C$ with the longer segment $v_1, P_i'[v_1, w_j], w_j, P[w_j, f_1], f_1, v_c$ yields a cycle in $H$ that is longer than $C$, a contradiction. Therefore we may assume that $P_1'$ and $P_2'$ are vertex-disjoint and edge-disjoint in $H$.

Next we show that

 \begin{equation}\label{h'P}
\mbox{\em no edge in $H'$ containing $u_{\ell+1}$ is in $P'_1$ or $P'_2$.}
\end{equation}

Indeed, suppose $h\in E(H')$ contains $u_{\ell+1}$. Then by the maximality of $\ell$ and Claim~\ref{xyh'2}, $h\subset V(P)$. Therefore, by the definition of $i_1,\ldots,i_\alpha$, $h\subseteq \{u_{i_1},\ldots,u_{i_\alpha}\}$. But such edges are not in $E(P'_1)\cup E(P'_2)$ by construction. This proves~\eqref{h'P}.

Observe that for $m\geq i_1$,  if $f_m$
  is in $P'_2$, then by the definition of $P'_2$, it must be the last edge of $P'_2$.

Let $a_1$ and $b_1$ be the first elements of $P_1'$ and $P_2'$ respectively.  Let $Q$ be a long $b_1, a_1$-segment of $C$ guaranteed by Claim~\ref{longQ} (recall that  if $P_1'$ is a path, 
  then $a_1$ is the first vertex of $P_1'$ and  if $P_1'$ is a partial path, then it is the first edge, and similar for $b_1$).
  
Next we show  
  \begin{equation}\label{i1-1}
  f_{i_1-1} \notin  E(P_2').
  \end{equation}

Suppose $f_{i_1-1} \in E(P_2')$. Since $P_2$ is aligned with $P'$, the segment $P_2'[b_1, f_{i_1-1}]$ does not intersect $P[u_{i_1}, u_{\ell+1}]$. Then \[b_1,Q,a_1, P_1', u_{i_1}, P[u_{i_1}, u_{\ell+1}], u_{\ell+1}, f_{i_1-1}, P_2'[f_{i_1-1}, b_1], b_1\] contains at least $c-(k-1)$ vertices in $C$ and all vertices in $S_1 \cup S_2 \cup \{u_{\ell+1}\} - \{u_{i_1-1}\}$, i.e., it has at least $c-(k-1) + k+1-1>c$ vertices. This proves~\eqref{i1-1}.

 Let $u_g$ be the last vertex of $P'_2$.
 
 {\bf Case 1.1.}  $f_{g-1}$ is the last edge of $P'_2$. Since $u_g\in Y$, $g-1\geq i_1$. Hence by Corollary~\ref{XY2},
 $f_{g-1}$ does not intersect $X$. Then $P_2'$ has at least two edges and at least one internal vertex, say $z$. By the definition of $P_2'$ and the fact that $f_{g-1} \neq e_c$, $z\notin X\cup Y\cup \{u_{i_1}\}$.
 Let $g'$ be the largest index less than $g-1$ such that $u_{g'}\in S_1\cup S_2$. If $g' \neq i_1-1$ (so $g' \geq i_1$),  consider
 \[C_1 = b_1,Q,a_1,P_1',u_{i_1},P[u_{i_1}, u_{g'}], u_{g'}, h'(u_{g'},u_{\ell+1}),u_{\ell+1}, P[u_{\ell+1}, u_g], u_g,P_2',b_1.\]
This cycle has at least $c-(k-1)$ vertices in $C$ and $z\notin X\cup Y$. 
The cycle $C_1$ also may miss at most two vertices in $S_1\cup S_2$ (namely $u_{i_1-1}$ and $u_{g-1}$), and since we are in Case 1, 
\[\left(S_1 \cup S_2 \cup \{u_{\ell+1}\} - \{u_{g-1}, u_{i_1-1}\}\right) \cap V(C) = \emptyset.\] 

Therefore $C_1$ contains at least \[|V(Q)| + |V(P_2') - (V(C) \cup V(P))| + |S_1 \cup S_2 \cup \{u_{\ell+1}\} - \{u_{g-1}, u_{i_1-1}\}| \geq c-(k-1) + 1 + k+1 - 2>c\] vertices, a contradiction.
%

If $g' = i_1-1$, then by~\eqref{i1-1}, $f_{g'} \notin E(P_2')$. First suppose $b_1 \neq v_c$. We let $Q'$ be a long $b_1, v_c$-segment of $C$ if $(C,P)$ is an o-lollipop or a long $b_1, e_c$-segment of $C$ if $(C,P)$ is a p-lollipop (without loss of generality, this path ends with $v_c$), and take the cycle
\[C_2 = b_1, Q', v_c (=u_1), \ldots, u_{i_1-1}, f_{i_1-1}, u_{\ell+1}, P[u_{\ell+1}, u_{g}], u_{g}, P_2', b_1\] which again omits  only $u_{g-1}$ from $S_1 \cup S_2$ (which is possibly in $V(C)$) and satisfies \[|C_2| \geq c-(k-1) + 1 + k+1 -2 > c.\]

Suppose now that $b_1 = v_c$. If $f_{i_1-1}  \in E(P_1')$, let
\[C_3 = v_c, Q, a_1, P_1'[a_1, f_{i_1-1}], f_{i_1-1}, u_{\ell+1}, P[u_{\ell+1}, u_g], u_g, P_2', v_c.\] 
Recall that $P_2'$ has an internal vertex $z\notin X\cup Y\cup \{u_{i_1}\}$. Also, all vertices in $S_1 \cup S_2 \cup \{u_{\ell+1}\} - \{u_{i_1-1}, u_{g-1}\}$ are in $C_3$ and none of them belongs to $C$. Therefore $|C_3| \geq c-(k-1) + 1 + (k+1) - 2 > c$. 

Lastly, if $f_{i_1-1} \notin E(P_1')$, then the cycle
\[C_4 = v_c, Q, a_1, P_1', u_{i_1}, f_{i_1-1}, u_{\ell+1}, P[u_{\ell+1}, u_g], P_2', v_c\]
 contains at least $c-(k-1) + k+1 - 1> c$ vertices. 

{\bf Case 1.2.} The last edge of $P'_2$ is not $f_{g-1}$. Then we let $g'$ be the largest index less than $g$ such that $u_{g'}\in S_1\cup S_2$. In this case, the cycle $C_1$ from the previous subcase can miss only $u_{i_1-1}$ in $S_1\cup S_2$ and contains at least $k-1$ vertices in $S_1 \cup S_2- \{u_{i_1-1}\}$ which are disjoint from $V(C)$.  
We get
\[|C_1| \geq |V(Q)| + |S_1 \cup S_2- \{u_{i_1-1}\}| + |\{u_{\ell+1}\}| \geq c-(k-1) + (k-1) + 1 > c.\]  This finishes Case 1.

\medskip
{\bf Case 2.}  The smallest $i$ with $u_i\in S_1\cup S_2$ is such that $u_i\in S_1$. Recall that in this case $i=i_1$
 and the other indices  of the vertices in  $ S_1 \cup S_2$ are $i_2,\ldots,i_{\alpha}$ in increasing order. 
 
 Now, define $P',P'_1,P'_2$ and $Q$ as in Case 1. Then we can repeat the final part of the proof of Case 1 with the simplification that the cycle $C_1$ omits at most the vertex $u_{g-1}$ in $S_1 \cup S_2$ and this occurs only if $f_{g-1}$ is the last edge of $P_2'$. As in Case 1.1, since $f_{g-1} \neq e_c$ and $g-1\geq i_1$, 
  $P_2'$ contains at least one vertex outside of $V(C) \cup Y\cup \{u_{i_1}\}$. Note also that it may be the case $u_{i_1}=u_1 \in V(C)$. 
 
 If no vertices of $S_1 \cup S_2$ are omitted from $C_1$, then 
 \[|C_1| \geq |V(Q)| + |S_1 \cup S_2| - |\{u_{i_1}\}| + |\{u_{\ell+1}\}| \geq c-(k-1) + k-1 + 1 > c.\]
 
 Otherwise, if $u_{g-1} \in S_1 \cup S_2$ was omitted from $C_1$, then
 \[|C_1| \geq |V(Q)| + |S_1 \cup S_2| - 1 -  |\{u_{i_1}\}| + |\{u_{\ell+1}\}| + |V(P_2') - V(C) \cup Y\cup \{u_{i_1}\}| > c.\]
This proves Part (A).

\medskip



Now we prove (B). Recall that $u_{j_1}, \ldots, u_{j_\beta}$ are the $H'$-neighbors of $u_{\ell+1}$ and suppose $\beta \geq k-1$. Corollary~\ref{XY2'} asserts that apart from $f_1,\ldots,f_{i_1-1}$ only $e_{k-1}$ may intersect both $X$ and $Z$. First part of our proof is to show that $e_{k-1}$ does not intersect $Z$.


\begin{claim}\label{xyf2} $f_{i} \subseteq V(P)$ for all $u_i \in N_{H'}(u_{\ell+1})$.
\end{claim}
\begin{proof}
By Claim~\ref{xyf3}, if $u \in f_{i} -V(P)$, then $u \notin V(C) \cup V(P)$. If $P$ is a path, then we can replace it with the longer path
$u_1, P[u_1, u_i], u_i, h'(u_i, u_{\ell+1}), u_{\ell+1}, P[u_{\ell+1}, f_i], f_i, u$.
Otherwise we replace $P$ with the partial path  $f_1, P[f_1, u_i], u_i, h'(u_i, u_{\ell+1}), u_{\ell+1}, P[u_{\ell+1}, f_i], f_i, u$.
\end{proof}

\begin{claim}\label{l>k-1} $\ell\geq k$.
\end{claim}
\begin{proof} Suppose  $\ell\leq k-1$. Since $|S_1| \geq k-1$,
 $\ell=k-1$ and each vertex in $V(P) - \{u_{\ell+1}\}$ including $u_1$ is in $S_1$.
 
 If $u_\ell\notin N_{H'}(u_{\ell+1})$, then $|N_{H'}(u_{\ell+1})|\leq \ell-1=k-2$, $N_{H'}(u_{\ell+1}) = V(P) - \{u_{\ell}, u_{\ell+1}\}$, and by Claims~\ref{xyc2'} and \ref{xyf2}, the only edges containing  $u_{\ell+1}$ and not contained in $V(P)$ could be $e_{k-1}$ and $f_\ell$.
 
We get $d_{H}(u_{\ell+1}) \leq {k-2 \choose r-1} + |E(P)| + 1 \leq {k-2 \choose r-1} + k$. When $r \geq 4$,  this is less than ${k-1 \choose r-1} + 1 = \delta(H)$. If $r = 3$, then each edge $f_i$ with $i < \ell$ containing $u_{\ell+1}$ satisfies $f_i = \{u_i, u_{i+1}, u_{\ell+1}\}$. Thus for $i \leq \ell-2$, such an edge $f_i$ is a subset of $N_{H'}[u_{\ell+1}]$ and is accounted for in the ${|N_{H'}(u_{\ell+1})| \choose r-1}$ term of $d(u_{\ell+1})$. Hence
 $d_H(u_{\ell+1}) \leq {k-2 \choose r-1} + |\{f_{\ell-1}, f_\ell, e_{k-1}\}| \leq {k-2 \choose r-1} + k-2 \leq {k-1 \choose r-1}$, a contradiction.
 
If $u_\ell\in N_{H'}(u_{\ell+1})$, then by Claim~\ref{xyf2}, $d_{H - E(C)}(u_{\ell+1}) \leq {k-1 \choose r-1}$, and the only edge containing  $u_{\ell+1}$ not contained in $V(P)$ could be $e_{k-1}$. Then in order  to satisfy $d(u_{\ell+1}) \geq \delta(H)$, we need that $u_{\ell+1}\in e_{k-1}$ and
every $r$-tuple contained in $V(P)$ and containing  $u_{\ell+1}$ is an edge in $H$. 

Hence we can reorder the vertices in $V(P)-u_1$ to make any vertex apart from $u_1$ the last vertex. The resulting path together with $C$ is also a best lollipop. By Claims~\ref{xyh'2}--\ref{xyf2} and the above, either $d(u_i) < {k-1 \choose r-1} + 1 \leq \delta(H)$ for some $i$ leading to a contradiction, or each $u_i \in V(P)$ is contained in $e_{k-1}$. In the latter case, 
$r=|e_{k-1}|\geq 2+\ell=k+1$, a contradiction.
\end{proof}

\begin{claim}\label{l+1C} $u_{\ell+1}\notin e_{k-1}$. 
\end{claim}
\begin{proof} Suppose $u_{\ell+1}\in e_{k-1}$. 
By Corollary~\ref{XY2'}, this is possible only if  $(C,P)$ is an o-lollipop and $c=2k-1$.
By Claim~\ref{l>k-1}, $\ell\geq k$ and so the cycle
$$C_0=v_c,e_{c-1},\ldots,v_{k},e,u_{\ell+1},f_{\ell},u_{\ell},\ldots,f_1, v_c$$
has at least $k+k=2k$ vertices, a contradiction.
\end{proof}

\begin{claim}\label{noedges} If $(C,P)$ is an o-lollipop, then  $e_{k-1}\cap Z=\emptyset$.
\end{claim}

\begin{proof} Suppose $u_i\in e_{k-1}\cap Z$.
By Corollary~\ref{XY2'}, this is possible only if  $c=2k-1$, $|S_1|=k-1$ and $j_1=1$.
By Claim~\ref{l+1C}, $i<\ell+1$. Let $i'$ be the largest number less than $i$ such that $u_{i'}\in N_{H'}(u_{\ell+1})$. Denote $I_1=\{i'+1,i'+2,\ldots,i-1\}$ and $I_2= \{1,\ldots,i'\}\cup \{i,i+1,\ldots,\ell\}$. By the choice of $i'$, $I_1 \cap S_1 = \emptyset$.

Consider the cycle
$$C_1=v_c,e_{c-1},\ldots,v_k,e_{k-1},u_i,f_i,\ldots,u_{\ell},f_\ell,u_{\ell+1},h'(u_{\ell+1},u_{i'}),f_{i'-1},\ldots,u_1.$$

We have $|C_1| \geq k + |S_1-\{u_1\}| + |\{u_{\ell+1}\}| \geq k + k-2 + 1 = 2k-1$ with equality only if $S_1 = \{u_i: i \in I_2\}$ and $|S_1| = |I_2|=k-1$. In particular, this means that the indices of the vertices in $S_1$ form  two intervals,
$\{1,\ldots,i'\}$ and $\{i,i+1,\ldots,\ell\}$, and the second of these intervals starts from $i$.
This yields $e\cap V(P)=\{u_i\}$.

Since we proved $|S_1 \cup S_2| \leq k-1 =  |S_1|$, we must have 
\begin{equation}\label{fi'}
\mbox{\em for each $m\in I_1$, $u_{\ell+1}\notin f_m$.}
\end{equation}

By Claim~\ref{xyf2}, for each $u_m \in N_{H'}(u_{\ell+1})$, $f_m \subseteq V(P)$. 
Suppose now that for some $m\in I_2 - \{i'\}$, and $i'<i''<i$, $u_{\ell+1},u_{i''}\in f_m$ and $H'$ has an edge $g$ containing $\{u_m,u_{m+1}\}$. In this case, if $m\geq i$, then the cycle
$$C_2=v_c,e_{c-1},\ldots,v_k,e_{k-1},u_i,f_i,\ldots,u_m,g,u_{m+1},f_{m+1},\ldots,u_{\ell},f_\ell,u_{\ell+1},f_m,u_{i''},f_{i''-1},\ldots,u_1$$
is longer than $C$, and if $1\leq m\leq i'$, then the cycle
$$C_3=v_c,e_{c-1},\ldots,v_k,e_{k-1},u_i,f_i,\ldots,u_{\ell},f_\ell,u_{\ell+1},f_m,u_{i''},f_{i''-1},\ldots,f_{m+1},u_{m+1},g,u_m,\ldots,u_1$$
is longer than $C$. Therefore,
\begin{equation}\label{fm4}
\parbox{14cm}{\em If $m\in I_2$, $i'<i''<i$ and $\{u_{\ell+1},u_{i''}\} \subset f_m$, then no edge in $H'$ contains $\{u_m,u_{m+1}\}$.}
\end{equation}

%

 

For $1\leq m\leq \ell$, call the edge $f_m$ {\em fitting} if $u_{\ell+1}\in f_m$ and $f_m\subseteq N_{H'}[u_{\ell+1}]$ and
{\em non-fitting} if $u_{\ell+1}\in f_m$ and $f_m\not\subseteq N_{H'}[u_{\ell+1}]$.
Let $R$ denote the set of fitting edges and $R'$ denote the set of non-fitting edges.
By~\eqref  {fi'}, if $f_m\in R$, then $m\in I_2$. By the definition of $I_2$, if 
 $m_1\in I_2- \{i',\ell\}$, then $m_1+1\in I_2$.

\medskip
{\bf Case 1.} $u_\ell\in N_{H'}(u_{\ell+1})$.  By Claim~\ref{l+1C}, all edges containing $u_{\ell+1}$ must either be contained in $S_1 \cup \{u_{\ell+1}\}$ or be non-fitting edges. Since $\delta(H)\geq 1+{k-1\choose r-1}$, this implies $R'\neq \emptyset$.
 Moreover, if there is a non-fitting edge $f_{m_1}\notin \{f_{i'},f_\ell\}$, then by Claim~\ref{xyf3} and~\eqref{fm4}, 
\begin{equation}\label{k-3}\hbox{the ${k-3\choose r-3}$ $r$-tuples  in the set $N_{H'}[u_{\ell+1}]$  containing $\{u_{m_1},u_{m_1+1},u_{\ell+1}\}$ are not edges of $H'$.
}
\end{equation}
The existence of such non-fitting $f_{m_1}\notin \{f_{i'},f_\ell\}$ is not possible if $r=3$ because in this case    $f_{m_1} = \{u_{\ell+1}, u_{m_1}, u_{m_1+1}\} \subseteq N_{H'}[u_{\ell+1}]$. So we may suppose $r \geq 4$. 
By~\eqref{k-3} we have
\[d(u_{\ell+1})\leq {|S_1| \choose r-1} - {k-3 \choose r-3} + |R'|  \leq {k-1 \choose r-1} - {k-3 \choose r-3} + |R'|.\]
Thus to have  $d(u_{\ell+1})\geq 1+{k-1\choose r-1}$, we need at least $1+{k-3\choose r-3}\geq k-2$   non-fitting edges $f_m$ for $m\in I_2$.

Since by the case,~\eqref{fm4} and Claim~\ref{xyf2}, $f_\ell$ has no vertex outside of $N_{H'}[u_{\ell+1}]$ and hence is  fitting, for each of the $k-2\geq r-1\geq 3$ values of  $m\in I_2-\{\ell\}$, $f_m$ must be non-fitting. But then at least  $1+{k-3\choose r-3}\geq k-2$ of the $r$-tuples contained in $N_{H'}[u_{\ell+1}]$ and containing $u_{\ell+1}$ are not edges of $H$. So $d(u_{\ell+1}) \leq {k-1 \choose r-1} - (k-2) + (k-2) < \delta(H)$.

Hence in order to have $d(u_{\ell+1}) \geq {k-1 \choose r-1} + 1$ we may assume that the only non-fitting edge is $f_{i'}$ and moreover every $r$-subset of $N_{H'}[u_{\ell+1}]$ containing $u_{\ell+1}$ is an edge of $H'$. In particular, there is an edge $g\in E(H')$ containing $u_{\ell+1},u_{i'}$ and $u_{i+1}$.
Then
\[C_4 = v_c, e_{c-1}, \ldots, v_k, e_{k-1}, u_i, f_{i-1}, \ldots, u_{i'+1}, f_{i'}, u_{\ell+1}, f_\ell, \ldots, u_{i+1}, g, u_{i'}, f_{i'-1}, \ldots, f_1, v_c\]
is longer than $C$.

\medskip
{\bf Case 2.} $u_\ell\notin N_{H'}(u_{\ell+1})$. Then $|N_{H'}(u_{\ell+1})|\leq k-2$,  and by Claim~\ref{l+1C}, 
\begin{equation}\label{k-2}
|R|+|R'|=d_P(u_{\ell+1})=d_H(u_{\ell+1})-d_{H'}(u_{\ell+1})\geq 1+{k-1\choose r-1}-\left[{k-2\choose r-1}-|R|\right]
=1+{k-2\choose r-2}+|R|.
\end{equation}
So, if $r\geq 4$, then $k\geq r+2\geq 6$, and $|R'|\geq 1+{k-2\choose 2}=1+\frac{(k-2)(k-3)}{2}\geq 1+\frac{3(k-2)}{2}$. 
But~\eqref{fi'} yields $d_P(u_{\ell+1})\leq |I_2|=k-1$,  a contradiction. On the other hand, if $r=3$, then 
 similarly to  Case 1, $R'\subseteq \{f_\ell,f_{i'}\}$, and hence~\eqref{k-2} yields
$$2\geq |R'|\geq 1+{k-2\choose r-2}=k-1\geq r+1=4,$$
 a contradiction.
\end{proof}

We now complete the proof of (B). As in the proof of (A), we apply Lemma~\ref{diraclemma} to the same path \[P' = v_1, e_1, v_2, \ldots, v_c, f_1, \ldots, f_{\ell}, u_{\ell+1}\] in $I_H$ with $z = u_{j_1}$. We obtain paths $P_1$ and $P_2$ and modify them to $P_1'$ and $P_2'$ with the same rules as in (A) but with $Z' = \{u_{j_1},f_{j_1}, u_{j_1+1},\ldots, u_{\ell+1}\}$ in place of $Y'$. 

We again get that $P_1'$ and $P_2'$ are vertex-disjoint and edge-disjoint and~\eqref{h'P} holds. Let $Q$ be a long segment of $C$ connecting $P_1'$ and $P_2'$ with at least $c-(k-1)$ vertices. Suppose the endpoints of $Q$ are the vertices $a_1$ and $b_1$.

Let $u_i$ be the last vertex of $P_2'$, and let $i'$ be the smallest index less than $i$ such that $u_{i'} \in S_1$. Consider the cycle
\[C' = a_1, Q, b_1, P_2', u_i, P[u_i, u_{\ell+1}], u_{\ell+1}, h'(u_{\ell+1}, u_{i'}), u_{i'} P[u_{i'}, u_{j_1}], u_{j_1}, P_1', a_1\]
which contains all $k$ vertices in $S_1 \cup \{u_{\ell+1}\}$. If none of these vertices is in $C$, then
$|C'| \geq c-(k-1) + k >c$, a contradiction. If there is such a vertex, it could be only $u_1$, in which case $(C,P)$ is an o-lollipop and  $j_1=1$. Then by Corollary~\ref{XY2'} and 
 Claim~\ref{noedges}, $P'_2$ contains at least one vertex outside of $V(C \cup P)$. It follows that $|C'| \geq c-(k-1) + k-1 + 1 >c$, a contradiction again.
\end{proof}




\section{Partial Berge paths in best $p$-lollipops are long}
In this section we concentrate on  $p$-lollipops and show that the partial path $P$ in them must be  long (namely, $\ell \geq k$). We do this by showing that $u_{\ell+1}$ has no $H'$-neighbors inside of $C$, and hence $P$ must be sufficiently long to contain all  $H'$-neighbors of $u_{\ell+1}$. The main lemma of this section is the following.

\begin{lemma}\label{partk} If $(C,P)$ is a p-lollipop then $|V(P)| = \ell \geq k$. 
\end{lemma}

\begin{proof}In Section~\ref{secl1} we showed that $\ell \geq 2$. Suppose towards contradiction that $2\leq \ell \leq k-1$. 
We will first show that

\begin{equation}\label{partialedges}\hbox{all $H'$-neighbors of $u_{\ell+1}$ are contained in $V(P)$.}
\end{equation}

By Claim~\ref{allneighbors}, all $H'$-neighbors of $u_{\ell+1}$ are in $V(C) \cup V(P)$. If $u_{\ell+1} \in e \in E(H')$ and $v_i \in e$ for some $v_i \in V(C)$, we  let $P' = v_i, e, u_{\ell+1}, P[u_{\ell+1}, u_2], u_2$. Observe that $V(P')-V(C) = V(P) - V(C)$, and $(C,P')$  is  better  than $(C,P)$ by Rule (d). This proves~\eqref{partialedges}.

Next we show that

\begin{equation}\label{edgesCP}\hbox{$u_{\ell+1}$ is contained in at least $k$ edges in $E(C) \cup E(P)$.}
\end{equation}

 By~\eqref{partialedges}, $|N_{H'}(u_{\ell+1})| \leq |V(P) - \{u_{\ell+1}\}| \leq k-2$.  Then the number of edges in $E(C) \cup E(P)$ containing $u_{\ell+1}$ must be at least \[\delta(H) - {|N_{H'}(u_{\ell+1})| \choose r-1} \geq {k-1 \choose r-1}+1 - {k-2 \choose r-1} \geq k-1,\] with equality only if $r = 3$, $|N_{H'}(u_{\ell+1})| = k-2$ (and so $V(P) = N_{H'}[u_{\ell+1}]$), $u_{\ell+1}$ is contained in all ${k-2 \choose r-1}$ possible $H'$-edges, and no edge of $E(C) \cup E(P)$ containing $u_{\ell+1}$ is a subset of $N_{H'}[u_{\ell+1}]$. If this is the case, then $e=\{u_2, u_3, u_{\ell+1}\} \in E(H') $, and we swap $f_2$ with $e$ to get a  partial path that is better than $P$ by Rule~(e). This proves~\eqref{edgesCP}. 

%

Say $|V(P)|=\ell = k-a$ where $2\leq k-a \leq k-1$. Since $|E(P)| = k-a$, by~\eqref{edgesCP}, $u_{\ell+1}$ is contained in at least $a$  edges in  $E(C)-e_c$. By
 Claim~\ref{allneighbors}(2), none of these edges is in the set $\{e_1, \ldots, e_{\ell-1}\} \cup \{e_{c-1}, \ldots, e_{c-(\ell-1)}\}$. Thus, $u_{\ell+1}$ is contained in at least $a$  edges in $\{e_{k-a}, e_{k-a+1}, \ldots, e_{c-(k-a)}\}$. Moreover, $u_{\ell+1}$ is contained in exactly $a$ such edges if and only if it is contained in all $k-a$ edges of $P$ (in particular, $u_{\ell+1} \in e_c$).

Let $e_i$ contain $u_{\ell+1}$ for some $i \neq c$. Consider the partial path $P' =e_i, u_{\ell+1}, f_\ell, \ldots, u_2$. Then $V(P') = V(P)$ and $E(P')-E(C)=E(P)-E(C)$. Thus $(C,P')$ also is a best lollipop.
So, as above we get  that all $H'$-neighbors of $u_2$ are  in $V(P)$, $u_{2}$ is contained in at least $k$ edges of $E(P') \cup E(C) = E(P) \cup E(C)$, and at least $a$ edges of $E(C)-\{e_i\}$ with equality only if $u_2 \in e_i$ by Claim~\ref{allneighbors}. Moreover, each of these edges is of distance at least $k-a$ from $e_i$.

Let $B_{i}$ be the set of edges of $E(C)$ containing $u_{i}$ for $i \in\{2,\ell+1\}$. Observe that $|B_i| \geq a+1$. Let $t = |B_2 \cap B_{\ell+1}|$. If $ t = 0$, let $e_\alpha, e_\beta, e_\gamma$ be edges such that $\alpha< \beta < \gamma$ (modulo $c$), $e_\alpha, e_\gamma \in B_{\ell+1}$, and $e_\beta \in B_{2}$. 
Then the segment from $e_\alpha$ to $e_{\gamma}$ in $C$ contains at least $2(k-a-1)$ edges not in $B_2 \cup B_{\ell+1}$ by Claim~\ref{shortpath}. We get
\[2k-1 \geq |E(C)|\geq |B_2| + |B_{\ell+1}| + 2(k-a-1) \geq 2(a+1) + 2(k-a-1) = 2k,\] a contradiction.

Now suppose $1\leq t \leq |B_2|$. Then surrounding each edge in $B_2 \cap B_{\ell+1}$ there are two intervals of $2(k-a-1)$ edges that are disjoint from $B_2 \cup B_{\ell+1}$. Moreover if there exists $e_\alpha \in B_2 -B_{\ell+1}$ and $e_\beta \in B_{\ell+1} - B_2$, then each pair of vertices in $(B_2 \cap B_{\ell+1}) \cup \{e_\alpha,e_\beta\}$ has distance at least $k-a$. In this case, there are at least $t+2$ intervals of $(k-a-1)$ edges not in $B_2 \cup B_{\ell+1}$.
  Therefore \[2k-1 \geq |E(C)| \geq |B_2 \cup B_{\ell+1}| + (t+2)(k-a-1)\geq 2(a+1)-t + (t+2)(k-a-1) = t(k-a-2) + 2k \geq 2k,\]
  a contradiction.  If $B_2 \subsetneq B_{\ell+1}$ or vice versa, then we have $t \geq |B_2| \geq a+1$. As before, for any $e_\beta \in B_{\ell+1} - B_2$, each pair of edges in $(B_2 \cap B_{\ell+1}) \cup \{e_\beta\}$ has distance at least $k-a$. So instead we get \[2k-1 \geq |B_{\ell+1} - B_2| + t + (t+1)(k-a-1) \geq 1+t + (t+1)(k-a-1) = (t+1)(k-a) \geq (a+2)(k-a).\]  But this does not hold when $a \geq 1$, $k \geq 3$, and $k-a \geq 2$.  
  
   The last case is $B_2 = B_{\ell+1}$. If $t \geq a+2$, then $2k-1 \geq t(k-a-1) + t \geq (a+2) (k-a)$, a contradiction again. So we consider the case where $t = |B_2| = |B_{\ell+1}| = a+1$. Because $B_{\ell+1}$ must contain $a$ edges within the at most $2a$ edges of $\{e_{k-a}, \ldots, e_{c-(k-a)}\}$ we must have $\ell=k-a= 2$ by Claim~\ref{shortpath}. Without loss of generality, we may assume that $B_2 = B_{\ell+1} = \{e_c, e_2, e_4, \ldots, e_{2k-4}\}$. We also have $r = |e_c| \geq |\{v_c, v_1, u_2, u_{\ell+1}\}| = 4$.

Suppose the edge $f_2 \in E(P)$ contains a vertex $v_i \in V(C)$. By Claim~\ref{bigsmallcycle}, $e_{i-1}, e_i$ cannot contain $u_{\ell+1}$. So we must have that $c=2k-1$ and $i=2k-2$. 
Therefore $f_2$ contains at least $r-1$ vertices outside of $V(C)$. As $\ell=2$, $|(f_2\cap V(P))-V(C)|\leq 2$.
So, since $r = |e_c| \geq 4$, there exists  $u \in f_2$ with $u \notin V(C) \cup V(P)$. 

By Claim~\ref{shortpath}, $u$ cannot belong to $e_i$ if $e_{i-1}\in B_2$ or $e_{i+1}\in B_2$. Hence
$\{e_i\in E(C): u\in e_i\}\subseteq B_2$. If an edge $e_i\in B_2$ contains all vertices in $f_2-V(C)$, then
$|e_i|\geq 2+|f_2|-1=r+1$, a contradiction. 
 Therefore some $u\in f_2$ is contained in at most $(|B_2|-1)+1=k-1$ edges of $E(C) \cup E(P)$, and hence
 $d_{H'}(u)\geq \delta(H)-(k-1) \geq 1$. 
  Say $u \in e \in E(H')$. If there exists $w \notin V(C)$ in $e$, then either $(C, e_c, u_2, f_2, u, e, w)$ or $(C, e_c, u_3, f_2, u, e, w)$ is a better lollipop. than $(C,P)$. So $e$ must contain $r-1$ vertices in $V(C)$. Without loss of generality, $v_i \in e$ and $e_i \in B_2$. Then replacing the segment $v_i,e_i, v_{i+1}$ in $C$ with $v_i, e, u, f_2, u_2, e_{i}, v_{i+1}$ yields a  cycle longer than $C$. 
\end{proof}

By applying Claim~\ref{allneighbors} and using that $\ell \geq k$, $c < 2k$, we obtain the following corollary.

   \begin{cor}\label{noC}If $(C,P)$ is a p-lollipop, then the only edge of $C$ that may contain $u_{\ell+1}$ is $e_c$.
   \end{cor}

\section{The paths in lollipops are short}
In this section we show that $P$ cannot be too long (namely, $\ell \leq k-2$). Our first step will be to show that if $u_{\ell+1}$ has $H'$-neighbors in $C$ and $P$ is long, then we can  find a better cycle than $C$. Then we apply Lemma~\ref{BL2} to analyze the case where all $H'$-neighbors of $u_{\ell+1}$ are in $P$. As a result of   Lemma~\ref{partk} and the lemma below, we obtain that $(C,P)$ is an $o$-lollipop.

 \begin{lemma}\label{bonl} If $k\geq r+2\geq 5$, then $\ell\leq k-2$.
 \end{lemma}

\begin{proof} Suppose $\ell\geq k-1$ and recall that by Lemma~\ref{partk} we have equality only if $(C,P)$ is an $o$-lollipop.

{\bf Case 1.}  Some $h\in E(H')$ contains $u_{\ell+1}$ and some $v_i\in V(C)-v_c$. By symmetry we may assume $i\leq c/2$ when $(C,P)$ is an o-lollipop and $i\leq (c+1)/2$ when $(C,P)$ is a p-lollipop. Consider the cycle
$$C_1=v_i, C[v_i,v_c],v_c,P,u_{\ell+1},h,v_i.$$
If $(C,P)$ is an o-lollipop, then $C_1$
has at least $(c-(c-2)/2)+k-1=c+\frac{2k-2-c+2}{2}>c$ vertices, a contradiction. 
If $(C,P)$ is a p-lollipop, then $C_1$
has at least $(c-(c-1)/2)+k=c+\frac{2k-c+1}{2}>c$ vertices,
 a contradiction again. This finishes Case 1.

 \medskip
 
 For $2\leq m\leq \ell+1$, let $B_m=\{e_j\in E(C): u_m\in e_j\}$ and $b_m=|B_m|$.
By Claim~\ref{allneighbors} and Corollary~\ref{noC}, 
\begin{equation}\label{bp}
\parbox{14cm}{\em if $b_{\ell+1}>0$,  then either 
$\ell=k-1$,
$(C,P)$ is an $o$-lollipop and $B_{\ell+1}=\{e_{k-1}\}$, or $(C,P)$ is a p-lollipop and $B_{\ell+1} = \{e_c\}$.}
\end{equation}

  Let $F=\{f_m\in E(P)-\{e_c\}:u_{\ell+1}\in f_m\}$. By Lemma~\ref{BL2}(A), $|F|\leq k-1$.

 \medskip
{\bf Case 2.} $N_{H'}(u_{\ell+1})\subset V(P)$. 
%
 By~\eqref{bp}, Lemma~\ref{BL2}(B) and the fact that $|F|\leq k-1$,
\begin{equation}\label{lower}
1+{k-1\choose r-1}\leq d(u_{\ell+1})\leq {|N_{H'}(u_{\ell+1})|\choose r-1}+|F|+b_{\ell+1}\leq 
{k-2\choose r-1}+(k-1)+1.
\end{equation}
For $r \geq 4$, regrouping, we get
$$k-1\geq {k-2\choose r-2}\geq  {k-2\choose 2}=\frac{(k-2)(k-3)}{2},$$
 yielding $k^2-7k+8\leq 0$, which is not true for $k\geq 6$. 
 This settles the case $r\geq 4$.
 
 So suppose $r = 3$. In particular, if $(C,P)$ is a p-lollipop, then $u_{\ell+1} \notin e_c = \{v_c, v_1, u_2\}$. Thus $b_{\ell+1} >0$ only if $(C,P)$ is an o-lollipop. If $|N_{H'}(u_{\ell+1})| \leq k-3$, then \[d(u_{\ell+1}) \leq {k-3 \choose 2} + k = {k-3 \choose 2} + {k-3 \choose 1} + 3 = {k-2 \choose 2} + 3 \leq {k-2 \choose 2} + {k-2 \choose 1} = {k-1 \choose 2},\]
 a contradiction.
 
 Hence  by Lemma~\ref{BL2}(B),   $|N_{H'}(u_{\ell+1})| = k-2$,       $u_\ell\in N_{H'}(u_{\ell+1})$, and
 $|S_1 \cup S_2| \leq k-1$. If $|F|+b_{\ell+1}\leq k-2$, then the RHS of~\eqref{lower} is at most ${k-1\choose r-1}$;
 so suppose $|F|+b_{\ell+1}\geq k-1$.
 
 By Claim~\ref{bestp}(A), for every $f_m\in F\setminus \{f_\ell\}$ the lollipop  $(C,P_m)$ where  $P_m$ is obtained from $P$ by replacing the subpath 
$u_m,f_m,u_{m+1},\ldots,u_{\ell+1}$ with the subpath $u_m,f_m,u_{\ell+1},f_\ell,u_\ell,\ldots,u_{m+1}$ also is a best lollipop.
Since $|F|\geq k-1-b_{\ell+1}\geq (r+2)-1-1= 3$ and  $ e_{k-1}$ may contain only one vertex of $P$, for some
$f_m\in F$,
 $u_{m+1}\notin e_{k-1}$ and hence by~\eqref{bp} $u_{m+1}$ does not belong to any edge of $C$. So we may assume that $b_{\ell+1}=0$. 
Then in view of~\eqref{lower}, if  $d(u_{\ell+1})\geq 1+{k-1\choose 2}$, then
\begin{equation}\label{allt}
\parbox{14cm}{\em 
 $|F|=k-1$, each $f_m \in F$ is not contained in $N_{H'}[u_{\ell+1}]$, and any two vertices in 
$N_{H'}(u_{\ell+1})$ form an edge of $H'$ together with $u_{\ell+1}$. }
\end{equation}
So, since $u_\ell\in N_{H'}[u_{\ell+1}]$,
$f_\ell\not\subset N_{H'}[u_{\ell+1}]$, but there is $g\in E(H')$ such that $\{u_{\ell},u_{\ell+1}\}\subset g$. 
Moreover, since $|N_{H'}(u_{\ell+1})| = k-2\geq r=3$, we can choose $g\subseteq N_{H'}(u_{\ell+1})-\{u_1\}$ and
\begin{equation}\label{bdeg}
\mbox{\em each vertex in $N_{H'}(u_{\ell+1})$ belongs to at least two edges of $H'$.}
\end{equation}
Then for $P'$ obtained from $P$ by replacing $f_\ell$ with $g$, the pair $(C,P')$ also is a best lollipop.

Suppose
$f_\ell=\{u_{\ell},u_{\ell+1},u\}$. By~\eqref{allt}, $u\notin N_{H'}[u_{\ell+1}]$. If $u\in V(C)-V(P)$, then we have Case 1 
for $(C,P')$, a contradiction.
 If $u\notin V(C)\cup V(P)$, then we can extend $P'$ by adding edge $f_\ell$ and vertex $u$.
 So, $u\in V(P)-N_{H'}[u_{\ell+1}]$. But then in view of~\eqref{bdeg}, the size of $N_{H'}(u_{\ell+1})$ corresponding to $(C,P')$ will be  $k-1$ because of the new vertex $u$, a contradiction.
  \end{proof}

Lemma~\ref{bonl} together with~Lemma~\ref{partk} yield

 \begin{cor}\label{bon3}  $(C,P)$ is an $o$-lollipop.
 \end{cor}

\section{Finishing proof of Theorem~\ref{mainthm}}

In this section we complete the proof of Theorem~\ref{mainthm}. One notable part of this section is that we construct another optimal lollipop in which the vertex $u_2$ plays the role of $u_{\ell+1}$. We consider the $H'$-neighborhoods of both $u_2$ and $u_{\ell+1}$ as well as the edges in $C$ containing these vertices, and we analyze how these sets can interact. We conclude that $u_2$ and $u_{\ell+1}$ cannot both have degree more than ${k-1 \choose r-1}$ without creating a cycle longer than $C$.

\bigskip

By Lemmas~\ref{ell3} and~\ref{bonl}, 
 $2 \leq \ell \leq k-2$. Corollary~\ref{bon3} gives that $(C,P)$ is an  o-lollipop.  The following lemma will be useful for bounding the size of $N_{H'}(u_{\ell+1})$.

\begin{lemma}\label{hamp2}
Let $ s+1 \geq b \geq 0$. Let $Q =v_0,v_1, \ldots, v_{s+1}$ be a graph path, and $I$ be a
non-empty independent subset of $\{v_1,\ldots, v_s\}$. If $B$ is a set of $b$ edges of $Q$ such that no edge in $B$
contains  any vertex in $I$, then $|I| \leq  \left\lceil\frac{ s-b}{2}\right\rceil$.

Moreover if $s-b$ is odd and $|I| = \frac{ s-b+1}{2}$, then for every $1\leq i \leq s$,  $v_i \in I$, or $e_i\in B$, or $e_{i-1} \in B$,  or $\{v_{i-1}, v_{i+1}\}  \subseteq I.$
\end{lemma}

\begin{proof} The claim is trivial if $b=0$ so assume $b>0$. 
 Iteratively contract all $b$ edges of $B$, say $Q'=v_0', v_1', \ldots, v'_{s+1-b}$ is the new path obtained. Observe that since $I$ was disjoint from the edges in $B$, after contraction $I$ is still an independent set in $Q'$ such that $I \subseteq \{v_1', \ldots, v_{s-b}'\}$. Therefore $|I| \leq \lceil \frac {s-b}{2} \rceil$. 
 
Now suppose $s-b$ is odd, $|I| = \left\lceil\frac{ s-b}{2}\right\rceil$, and for some $i$, $e_i, e_{i-1} \notin B$.
If without loss of generality $v_{i+1} \notin I$, then we contract the edge $v_iv_{i+1}$ and apply the result to the new path, $I$, and $B$ to obtain $I \leq \lceil \frac{s-1-b}{2} \rceil  < \lceil \frac{s-b}{2} \rceil$.
\end{proof}

For $i\in \{2,\ell+1\}$, let $A_i = N_{H''}(u_i) \cap V(C)$ and $B_i$ be the set of  edges in $E(C)$ containing $u_i$. Also, let $a_i = |A_i|$ and $b_i = |B_i|$. Let $F=\{f_m: u_{\ell+1}\in f_m\}$. We will heavily use the fact that
\begin{equation}\label{basic}
1+{k-1\choose r-1}\leq d_H(u_{\ell+1})=b_{\ell+1}+|F|+d_{H'}(u_{\ell+1}).  
\end{equation}


 \begin{claim}\label{bon2} If  $\ell\geq 2$, then some edge $e\in E(H'')$
 containing $u_{\ell+1}$ intersects $C$. \end{claim}

\begin{proof} 
  
  If the claim fails,  then $|N_{H''}(u_{\ell+1})| \leq |V(P) - V(C)| = \ell-1$ and $f_1\notin F$. 
  Hence using  Claim~\ref{allneighbors}, 
  \begin{equation}\label{degn}
   1+{k-1\choose r-1}\leq d_H(u_{\ell+1})\leq {|N_{H'}(u_{\ell+1})| \choose r-1} + (c-2\ell) + \ell-1 \leq {\ell-1\choose r-1}+c-\ell-1.
   \end{equation}
 Since $\ell\leq k-2$ and the function $h(\ell):= {\ell-1\choose r-1}+c-\ell-1$ does not decrease for integers $\ell\geq r-1$,
 in the range $r-1\leq \ell\leq k-2$,~\eqref{degn} gives
 $$1+{k-1\choose r-1}\leq {(k-2)-1\choose r-1}+(2k-1)-(k-2)-1={k-3\choose r-1}+k,
  $$
 which is not true for $k\geq r+2\geq 5$. 

Otherwise  ${\ell-1\choose r-1}=0$, so~\eqref{degn} yields $1+{k-1\choose 2}\leq (2k-1)-2-1$, which is not true for $k\geq 5$.
 \end{proof}

Fix an edge $f_0\in E(H'')$ containing $u_{\ell+1}$ and some $v_j\in V(C)$ provided
by Claim~\ref{bon2} (see Figure~\ref{broken}). Possibly, $f_0=f_1$.
Consider path $P' = v_j, f_0, u_{\ell+1}, f_\ell, \dots, f_2, u_2$. Since $V(P)-V(C) = V(P')-V(C)$ and $f_1 \nsubseteq V(P)-V(C)$, $(C, P')$ is also a best lollipop. Thus many arguments we apply to $u_{\ell+1}$ will also apply symmetrically to $u_2$. 

Let $F'=\{f_m: m\in \{0,2,3,\ldots,\ell\}\;\mbox{\em and}\; u_{2}\in f_m\}$. 
 If $r=3$, then for $\ell \geq 3$, not all of $f_1,f_2$ and $f_3$ contain  $\{u_2,u_{\ell+1}\}$. So, we may assume
\begin{equation}\label{d27}
\mbox{\em if $r=3$, then $|F|\leq \max\{2,\ell-1\}$.}
\end{equation}

 \begin{figure}
 \centering
 \includegraphics[scale=.6]{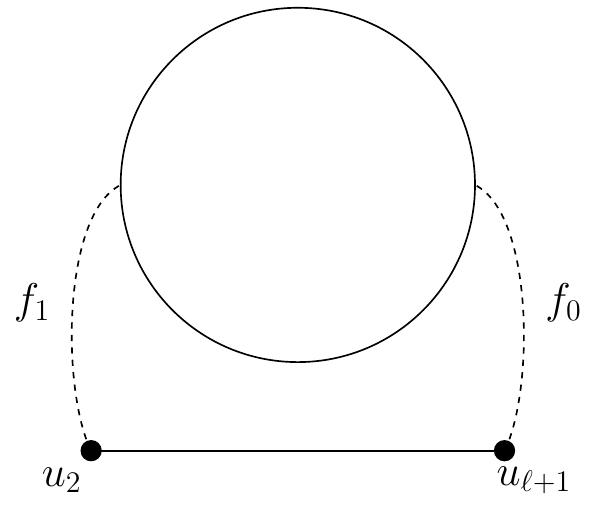}
 \caption{An o-lollipop $(C,P)$ with edge $f_0$ containing $u_{\ell+1}$.}
 \label{broken}
 \end{figure}

\begin{claim}\label{a2al+1}
$A_{\ell+1}=A_2$.
\end{claim}
\begin{proof} Suppose $A_{\ell+1}\neq A_2$. By symmetry, we may assume $A_2-A_{\ell+1}\neq \emptyset$. We may rename the vertices in $C$ and edges in $H''$ so that $v_c\in A_2-A_{\ell+1}$ and $f_1$ contains $v_c$ and $u_2$. This new lollipop (we still call it $(C,P)$) remains a best lollipop. So by Claim~\ref{allneighbors}, $N_{H'}(u_{\ell+1})\subseteq \{v_{\ell+1},v_{\ell+2},\ldots,v_{c-\ell-1}\}\cup \{v_c\}$. 
By Claim~\ref{allneighbors}, $B_{\ell+1}\subseteq \{e_\ell,e_{\ell+1},\ldots,e_{c-\ell-1}\}$.
 By Claim~\ref{bigsmallcycle}, if $e_i \in B_{\ell+1}$, then $v_i, v_{i+1} \notin N_{H'}(u_{\ell+1})$, and $N_{H'}(u_{\ell+1}) \cap V(C)$ does not contain two consecutive vertices of $C$. Hence, remembering that $v_c\notin A_{\ell+1}$, we apply
  Lemma~\ref{hamp2} to the (graph) path $Q'=v_{\ell},\ldots,v_{c-\ell}$, $I=A_{\ell+1}$ and $B'=B_{\ell+1}$, and get
 \begin{equation}\label{nbound}
a_{\ell+1} \leq   \Big\lceil \frac{(c-2\ell-1)-b_{\ell+1}}{2}\Big\rceil =\Big\lceil\frac{c-1-b_{\ell+1}}{2}\Big\rceil-\ell\leq
k-1-\ell-\Big\lfloor\frac{b_{\ell+1}}{2}\Big\rfloor.
\end{equation}
Since $u_{\ell+1}\notin f_1$, $|F|\leq \ell-1$. Since $N_{H'}(u_{\ell+1})\subseteq A_{\ell+1}\cup V(P) - \{u_{\ell+1}\}$ and $|V(P)-V(C)-u_{\ell+1}|=\ell-1$,    $|N_{H'}(u_{\ell+1})| \leq a_{\ell+1} + \ell -1$. Combining this with~\eqref{nbound} and~\eqref{basic}, we get
\begin{equation}\label{deg3}
1+{k-1\choose r-1}\leq b_{\ell+1} + (\ell-1) +{k-2-\left\lfloor{b_{\ell+1}}/{2}\right\rfloor\choose r-1}.
\end{equation}
For fixed $k,r,\ell$ satisfying the theorem, the maximum of the sum $b_{\ell+1} +{k-2-\left\lfloor{b_{\ell+1}}/{2}\right\rfloor\choose r-1}$ over nonnegative $b_{\ell+1}$ such that $k-2-\left\lfloor{b_{\ell+1}}/{2}\right\rfloor\geq r-1$ is achieved at 
$b_{\ell+1}=1$. Hence~\eqref{deg3} yields
$1+{k-1\choose r-1}\leq 1+(\ell-1)+{k-2\choose r-1},$
which in turn gives ${k-2\choose r-2}\leq \ell-1\leq k-3$, a contradiction.
 \end{proof}
 
 In view of this claim, let $A=A_{\ell+1}=A_2$ and $a=|A|$. Since $v_c\in A$, 
 instead of~\eqref{nbound} we have 
\begin{equation}\label{nbound2}
a \leq  1+ \Big\lceil \frac{(c-2\ell-1)-b_{\ell+1}}{2}\Big\rceil =1+\Big\lceil\frac{c-1-b_{\ell+1}}{2}\Big\rceil-\ell\leq
k-\ell-\Big\lfloor\frac{b_{\ell+1}}{2}\Big\rfloor.
\end{equation}

 \begin{claim}\label{a+l-1} $|N_{H'}(u_{\ell+1})|= a+\ell-1$, i.e., $N_{H'}[u_{\ell+1}]=A_{\ell+1}\cup V(P)$.
 \end{claim}
 
 \begin{proof} If $|N_{H'}(u_{\ell+1})|\leq  a+\ell-2$, then by~\eqref{basic} and~\eqref{nbound2},
 \begin{equation}\label{basic3}
1+{k-1\choose r-1}\leq b_{\ell+1}+|F|+{|N_{H'}(u_{\ell+1})|\choose r-1}\leq b_{\ell+1}+|F|+{k-2-\lfloor{b_{\ell+1}}/{2}\rfloor\choose r-1}.  
\end{equation}
For fixed $k,r,\ell,|F|$ satisfying the theorem, the maximum of the RHS of~\eqref{basic3} over suitable $b_{\ell+1}$  is achieved at 
$b_{\ell+1}=1$. Hence~\eqref{basic3} yields $1+{k-1\choose r-1}\leq 1+|F|+{k-2\choose r-1}$, i.e. ${k-2\choose r-2}\leq |F|$.  
Since $\ell\leq k-2$, for $r=3$ by~\eqref{d27}, this gives ${k-2\choose 1}\leq \max\{2,\ell-1\}\leq k-3$, an impossibility,
 and for $r\geq 4$ this yields ${k-2\choose 2}\leq \ell\leq k-2$, which is not true for $k\geq r+2\geq 6$.
\end{proof}

\begin{claim}\label{AA}
For all $v_j,v_{j'}\in A$, either $j'=j$ or $|j'-j|>\ell$ (modulo $c$).
\end{claim}

\begin{proof} Suppose the claim fails. By symmetry, we may assume $v_c,v_j\in A$ and $1\leq j\leq \ell$. 


By Claim~\ref{shortpath}(3), if there exists $e,f \in E(H'')$ such that $\{v_c, u_2\} \subset e$ and $\{v_j, u_{\ell+1}\} \subset f$, then $f=e$. Thus the only edge of $H''$ containing $v_c$ or $v_j$ and $u_{\ell+1}$ is $f_1$. Hence
$d_{H'}(u_{\ell+1})\leq 1+ {|N_{H'}(u_{\ell+1})|-2 \choose r-1}$.
Therefore, by~\eqref{nbound2} instead of~\eqref{deg3} we get
\begin{equation}\label{deg3'}
1+{k-1\choose r-1}\leq b_{\ell+1} + \ell+1 +{k-3-\left\lfloor{b_{\ell+1}}/{2}\right\rfloor\choose r-1}.
\end{equation}
For fixed $k,r,\ell$ satisfying the theorem, the maximum of the sum ${k-3-\left\lfloor{b_{\ell+1}}/{2}\right\rfloor\choose r-1}+b_{\ell+1} $ over suitable $b_{\ell+1}$  is achieved at 
$b_{\ell+1}=1$. Hence~\eqref{deg3'} together with $\ell\leq k-2$ yields
$1+{k-1\choose r-1}\leq 1+(\ell+1)+{k-3\choose r-1}\leq k+{k-3\choose r-1},$
which  is not true when $k\geq r+2\geq 5$.
 \end{proof}

\begin{claim}\label{PN'}
Each $f_i \in F$ is contained in $N_{H'}[u_{\ell+1}]$. 
\end{claim}

\begin{proof} 
Assume there exists $f_i$ not contained in $N_{H'}[u_{\ell+1}]$. Since $A_2\supseteq f_1\cap V(C)$, $A_2\neq \emptyset$. So,
by Claims~\ref{a2al+1} and~\ref{a+l-1},
$d_{H'}(u_{\ell+1})> 0$.
Let $w\in f_i-N_{H'}[u_{\ell+1}]$. Also by Claim~\ref{a+l-1}, $u_i \in N_{H'}(u_{\ell+1})$. Suppose first $i=1$. Let $v_j\in A_{\ell+1}$. By Claim~\ref{a+l-1}, there is an edge $h\in E(H')$ containing $\{u_{\ell+1},v_j\}$. Then path
$P_0=v_j,h,u_{\ell+1},f_{\ell},u_{\ell},\ldots,u_2,f_1,w$ is longer than $P$, a contradiction.


Suppose now $i\geq 2$.
Let $g\in E(H')$ contain $\{u_{\ell+1},u_i\}$. Let $P_1=u_1,f_1,\ldots,u_i,g,u_{\ell+1},f_{\ell},u_{\ell},\ldots,u_{i+1}$. Since $V(P_1)-V(C)=V(P)-V(C)$ and $f_i\not\subset V(P)$, the lollipop $(C,P_1)$ is a best lollipop.
If $w\notin V(C)$, then by appending to $P_1$ edge $f_i$ and vertex $w$ we get a better lollipop, a contradiction.
So, $w\in V(C)-A$, say $w=v_j$. Let $P_2=v_j,f_i,u_{i+1},f_{i+1},\ldots,u_{\ell+1},g,u_i,f_{i-1},\ldots,u_2$. Again,
$(C,P_2)$ is a best lollipop. Define $H_2'$ to be the hypergraph with $E(H_2') = E(H) - E(C) - E(P_2)$, and $H_2'' = H_2' + f_i$. Note that $H_2'$ and $H_2''$ play the role of $H'$ and $H''$ respectively for the best lollipop $(C,P_2)$. Moreover, define $A'_2 = N_{H_2''}(u_2) \cap V(C)$ ($A'_2$ plays the role of $A_{\ell+1}$). Then many of the claims we proved for $(C,P)$ also apply to $(C,P_2)$. Namely, $N_{H_2'}[u_2]= A'_2 \cup V(P_2) =  A'_2 \cup V(P)$ by Claim~\ref{a+l-1}. Since $v_j \in f_i$, we have $v_j \in A_2'$, so $v_j \in N_{H_2'}(u_2)$ and there exists some edge in $ E(H) - E(C) - E(P_2)$ containing both $u_2$ and $v_j$. Since $E(H) - E(C) - E(P_2) \subseteq E(H'')$, This implies that $v_j \in A_2 = A$, a contradiction.
\end{proof}
By Claim~\ref{PN'}, instead of~\eqref{basic} we have
\begin{equation}\label{basic2}
d(u_{\ell+1}) \leq {|N_{H'}(u_{\ell+1})| \choose r-1} + b_{\ell+1}  \leq {a + \ell-1 \choose r-1} + b_{\ell+1} \leq {k-1-\lfloor b_{\ell+1}/2 \rfloor \choose r-1} + b_{\ell+1}.
\end{equation}

\begin{claim}\label{b=1}$b_{\ell+1} = 1$ and $|N_{H'}(u_{\ell+1})| = k-1$.
\end{claim}

\begin{proof}
If $b_{\ell+1} = 0$, then by~\eqref{basic2}, $d(u_{\ell+1}) \leq {k-1 \choose r-1} < \delta(H)$. On the other hand, if $b_{\ell+1} \geq 2$ then the maximum of the RHS of~\eqref{basic2} is achieved at $b_{\ell+1} =3$, and so is
 at most ${k-1 - 1 \choose r-1} + 3 \leq {k-2 \choose r-1} + k-2 < \delta(H).$ This proves $b_{\ell+1}=1$. In view of this, if $|N_{H'}(u_{\ell+1})| \leq k-2$, then $d(u_{\ell+1}) \leq {k-2 \choose r-1} + 1 < \delta(H)$.
\end{proof}
By Claims~\ref{a+l-1} and~\ref{b=1}, we must have \[a-1 = |A - V(P)| = k-1 - \ell = \lceil \frac{(2k-1) -1 - b_{\ell+1}}{2} \rceil - \ell = \lceil \frac{(2k-1) -1 -1}{2} \rceil-\ell.\]  

We apply Lemma~\ref{hamp2} to the (graph) path $v_1, v_2,  \ldots, v_{c-1}$, with $I  = A - V(P)$, $B = B_{\ell+1}$, and $c-1 = 2k-2$. In particular since the numerator $2k-3$ is odd, the ``equality" part of Lemma~\ref{hamp2} holds. That is, 
\begin{equation}\label{equality1}
\hbox{for every $1\leq i \leq c-1$,   $v_i \in A- V(P)$, or $e_i \in B_{\ell+1}$, or $e_{i-1} \in B_{\ell+1}$ or $\{v_{i-1}, v_{i+1}\} \in A - V(P)$. }
\end{equation}We now complete the proof of Theorem~\ref{mainthm} by showing that for some $i$, ~\eqref{equality1} does not hold.

Since $|V(P) - \{u_{\ell+1}\}| = \ell \leq k-2$ and $|N_{H'}(u_{\ell+1})| =k-1$, $A$ contains some $v_j \in V(C) - \{v_c\}$. By Claim~\ref{allneighbors}, $j \in \{\ell+1, \ldots, c-\ell-1\}$.  By symmetry, we may assume $j < c-\ell-1$. 

We now show that~\eqref{equality1} does not hold for  $i=j+1$. By Claim~\ref{AA}, $v_{j+1}$ and $v_{j+2}$
are not in $A-V(P)$.
By Claim~\ref{bigsmallcycle}(a), $e_{j}, e_{j+1} \notin B_{\ell+1}$. 
 Thus~\eqref{equality1} fails, completing the proof of Theorem~\ref{mainthm}.

\bigskip
\bigskip

{\bf Acknowledgment.} We thank Sam Spiro for his helpful discussions on this topic and the referees for valuable comments.


\begin{thebibliography}{99}
\small


\bibitem{BGHS}
J.-C. Bermond, A. Germa, M.-C. Heydemann, D. Sotteau, Hypergraphes Hamiltoniens,
in \emph{Probl\`emes combinatoires et th\`eorie des graphes} (Colloq. Internat.
CNRS, Univ. Orsay, Orsay, 1976). Colloq. Internat. CNRS,  \textbf{260} (CNRS,
Paris, 1978), 39--43.
%

\bibitem{CEP}
D. Clemens, J. Ehrenm\" uller, and Y. Person, A Dirac-type theorem for Hamilton
Berge cycles in random hypergraphs, {\em Electron. J. Combin.} \textbf{27} (2020), no. 3, Paper No. 3.39, 23 pp.


\bibitem{CP}
M. Coulson, G. Perarnau, 
A Rainbow Dirac's Theorem, 
{\em SIAM J. Discrete Math.} \textbf{34} (2020), 
  1670--1692.


\bibitem{D}
G. A. Dirac,
Some theorems on abstract graphs, \emph{Proc. London Math. Soc. (3)} \textbf{2}
(1952), 69--81.
%

\bibitem{enom}
H. Enomoto, Long paths and large cycles in finite graphs, {\em J. Graph Theory} 8 (1984), 287--301.


%
%

\bibitem{EG}
P. Erd\H{o}s, T. Gallai,
On maximal paths and circuits of graphs, \emph{Acta Math. Acad. Sci. Hungar.} 
\textbf{10} (1959), 337--356.



\bibitem{EGMSTZ}  B. Ergemlidze, E. Gy\H ori, A. Methuku,  N. Salia, C. Tompkins, and O. Zamora,
    Avoiding long Berge cycles: the missing cases $k=r+1$ and $k=r+2$. {\em Combin. Probab. Comput.} \textbf{29} (2020),  423--435.
    

\bibitem{Fan1} G. Fan, 
Long cycles and the codiameter of a graph. I,  {\em J. Combin. Theory Ser. B} 49 (1990),  151--180. 

\bibitem{Fan2} G. Fan,  Long cycles and the codiameter of a graph. II,  Cycles and rays (Montreal, PQ, 1987), 87--94, NATO Adv. Sci. Inst. Ser. C: Math. Phys. Sci., 301, Kluwer Acad. Publ., Dordrecht, 1990.

%
%
\bibitem{FKLdirac} Z. F\"uredi, A. Kostochka, R. Luo, Berge cycles in non-uniform hypergraphs,  {\em  Electronic J. Combin.}  \textbf{27} (2020), Paper No. 3.9, 13 pp.

\bibitem{FKL}  Z. F\"uredi, A. Kostochka, R. Luo:
Avoiding long Berge cycles,
\emph{J.~Combinatorial Theory, Ser.~B}  {\bf 137} (2019),  55--64.


\bibitem{FKL2conn} Z. F\"uredi, A. Kostochka, R. Luo, On $2$-connected hypergraphs with no long cycles,  {\em  Electronic J. Combin.}  \textbf{27} (2019), Paper No. 4.26, 36 pp.

\bibitem{GKL}   E. Gy\H ori, G. Y. Katona and N. Lemons,  Hypergraph extensions of the Erdős-Gallai theorem. European J. Combin. 58 (2016), 238--246.


\bibitem{GLSZ}   E. Gy\H ori, N. Lemons,, N. Salia,  and O. Zamora,
 The structure of hypergraphs without long Berge cycles. J. Combin. Theory Ser. B 148 (2021), 239--250.



\bibitem{GSZ}   E. Gy\H ori,  N. Salia,  and O. Zamora,
 Connected hypergraphs without long Berge-paths. European J. Combin. 96 (2021), Paper No. 103353, 10 pp.



%
%
%
%
%
%
%

\bibitem{jackson2} B. Jackson, Long cycles in bipartite graphs, J. Combin. Theory, Ser. B, 38 (1985),
118--131.

 \bibitem{KL1} A. Kostochka and R. Luo,
 On $r$-uniform hypergraphs with circumference less than $r$,
{\em Discrete Appl. Math.} \textbf{276} (2020), 69--91.



\bibitem{KLM}
A. Kostochka, R. Luo, G. McCourt,
Dirac's Theorem for hamiltonian Berge cycles in uniform hypergraphs, submitted, 
{\tt  arXiv:2109.12637}, (2022), 23 pp.


\bibitem{KLM3}
A. Kostochka, R. Luo, G. McCourt, Minimum degree ensuring that a hypergraph is
hamiltonian-connected, submitted, 
{\tt  arXiv:2207.14794}, (2022), 22 pp.

\bibitem{KLM2}  A. Kostochka, R. Luo, G. McCourt,  On a property of $2$-connected graphs and Dirac's Theorem,
 submitted,  {\tt  arXiv:2212.06897}, 4 p, 2022. 



%
%
%
\bibitem{MHG}
Y. Ma, X. Hou, J. Gao,
A Dirac-type theorem for uniform hypergraphs, 
{\tt  arXiv:2004.05073}, (2020), 17 pp.



%
%
%
%
\bibitem{SN} N. Salia, P\' osa-type results for Berge hypergraphs,  {\tt arXiv:2111.06710v2},  15 p, 2021.


\bibitem{VZ}
H. J. Voss, C. Zuluaga, Maximale gerade und ungerade Kreise in Graphen I, Wiss 2.
Tech. Hochsch. Ilmenau 23, No. 4 (1977), 57--70.


\end{thebibliography}
\end{document}